\documentclass[10pt]{amsart}

\usepackage{amsmath}
\usepackage{amssymb}
\usepackage{verbatim}
\usepackage{enumerate}

\usepackage{setspace}
\usepackage{fancyhdr}
\usepackage{indentfirst}
\usepackage{amsthm}
\usepackage{mathrsfs}
	
\theoremstyle{definition}
\newtheorem{thm}{Theorem}[section]
\newtheorem{lem}[thm]{Lemma}
\newtheorem{prp}[thm]{Proposition}
\newtheorem{dfn}[thm]{Definition}
\newtheorem{cor}[thm]{Corollary}

\newtheorem{rmk}[thm]{Remark}
\newtheorem{ntn}[thm]{Notation}

\newcommand{\al}{\alpha}
\newcommand{\bt}{\beta}
\newcommand{\gm}{\gamma}
\newcommand{\dt}{\delta}

\newcommand{\eps}{\varepsilon}

\newcommand{\ld}{\lambda}
\newcommand{\sm}{\sigma}

\newcommand{\ph}{\varphi}

\newcommand{\om}{\omega}

\newcommand{\Z}{{\mathbb{Z}}}

\newcommand{\C}{{\mathbb{C}}}
\newcommand{\N}{{\mathbb{N}}}

\newcommand{\id}{{\mathrm{id}}}

\newcommand{\sint}{{\mathrm{int}}}
\newcommand{\diam}{{\mathrm{diam}}}

\newcommand{\spec}{{\mathrm{sp}}}

\newcommand{\supp}{{\mathrm{supp}}}

\newcommand{\Aut}{{\mathrm{Aut}}}

\newcommand{\Her}{\mathrm{Her}}

\newcommand{\dirlim}{\displaystyle \lim_{\longrightarrow}}

\newcommand{\ts}[1]{{\textstyle{#1}}}

\newcommand{\norm}[1]{\left\Vert#1\right\Vert}
\newcommand{\abs}[1]{\left\vert#1\right\vert}

\newcommand{\set}[1]{\left\{#1\right\}}
\newcommand{\ten}{\otimes}
\newcommand{\del}{\partial}

\newcommand{\JS}{\mathcal{Z}}

\title[The Tracial Quasi-Rokhlin Property]{Crossed Products
by Automorphisms with the Tracial Quasi-Rokhlin Property}

\author{Julian Buck}

\address{Department of Mathematics, Francis Marion University, 
Florence, SC 29502.}

\thanks{This research was part of the author's Ph.D. thesis at the
University of Oregon, completed under the direction of N. 
Christopher Phillips}
\date{27 June 2013}

\begin{document}

\begin{abstract}

We introduce the {\emph{tracial quasi-Rokhlin property}} for an
automorphism $\al$ of a unital $C^{*}$-algebra $A$, which is not
assumed to be simple. We show that under suitable hypotheses, 
the associated crossed product $C^{*}$-algebra $C^{*}(\Z,A,\al)$ is
simple, and there is a bijection between the space of tracial states
on $C^{*}(\Z,A,\al)$ and the $\al$-invariant tracial states on $A$.
We show that, for a minimal dynamical system $(X,h)$ and a 
simple, separable, unital $C^{*}$-algebra $A$, the automorphism 
$\bt$ which extends the action of $h$ on $C(X)$ has the tracial 
quasi-Rokhlin property, and hence that $C^{*}(\Z,C(X,A),\bt)$ has 
the structural properties described above.
\end{abstract}

\maketitle 

\section{Introduction}

\indent

The study of $C^{*}$-algebras arising through crossed product
constructions has been an area of significant interest in the
Elliott classification program for nuclear $C^{*}$-algebras. 
Two areas where considerable success has been achieved are 
crossed products associated to minimal dynamical systems and 
crossed products by automorphisms with various forms of the 
Rokhlin property. In the first situation, the case of Cantor minimal 
systems was studied extensively (see for example \cite{GPS} and 
\cite{Pt1}), and the techniques of \cite{Pt1} were later considerably 
generalized first to minimal diffeomorphisms of finite-dimensional 
compact manifolds (see the long unpublished preprint 
\cite{QLinPhDiff}, and also the survey articles \cite{QLinPh1} and 
\cite{QLinPh2}), and later to finite-dimensional compact metric 
spaces (see \cite{HLinPh} and \cite{TomsWinter} for the best 
known results). In the second situation, see for example \cite{Iz}, 
\cite{Ks1}, \cite{Ks2}, \cite{Ks3}, and \cite{LO} for results related to 
the Rokhlin property, and \cite{ArTrpPrfree}, \cite{ELPW}, 
\cite{OPTrp}, and \cite{PhTrpfg} for results related to various forms 
of the tracial Rokhlin property.

There is little existing overlap between these two branches of
research into crossed products. Most forms of the Rokhlin and tracial 
Rokhlin properties are formulated for $C^{*}$-algebras containing 
many projections (such as in the real rank zero case), while the 
$C^{*}$-algebra $C(X)$ may have few or no non-trivial projections. 
We introduce the {\emph{tracial quasi-Rokhlin property}} for 
automorphisms of a unital, separable $C^{*}$-algebra $A$ which is 
not assumed to be simple nor contain any non-trivial projections. In 
fact, the $C^{*}$-algebras in which we will be most interested will be 
of the form $C(X,A)$, where $X$ is an infinite compact metrizable 
space and $A$ is a simple, separable, unital, infinite-dimensional 
$C^{*}$-algebra.

In Section \ref{SectionTQRP}, we define the tracial quasi-Rokhlin 
property, and show that if $\al$ is an automorphism of $A$ and $A$ 
has no non-trivial $\al$-invariant ideals, then the crossed product 
$C^{*}(\Z,A,\al)$ is simple. Further, an additional technical 
assumption about $A$ (which is satisfied for our main algebras of 
interest) allows us to also show that the restriction mapping 
$T(C^{*}(\Z,A,\al) )\to T_{\al}(A)$, between the simplex of tracial 
states on the crossed product and the simplex of $\al$-invariant 
tracial states on $A$, is a bijection. 

In Section \ref{SectionAutomsWithTQRP} we use this condition to 
show that (with appropriate hypotheses on $X$ and $A$) certain 
automorphisms $\bt$ of the algebra $C(X,A)$, which act minimally 
on the center $C(X)$, have the tracial quasi-Rokhlin property. After 
examining the structure of ideals in $C(X,A)$ and of its tracial 
state space, it will follow that the structural theorems of Section 
\ref{SectionTQRP} apply the the associated crossed product 
$C^{*}$-algebras $C^{*}(\Z,C(X,A),\bt)$.

We would like to thank N. Christopher Phillips for his numerous 
contributions and insights as this research was completed as part 
of the author's Ph.D. thesis under his supervision. We would also 
like to thank Dawn Archey, George Elliott, Huaxin Lin, and Efren 
Ruiz for helpful suggestions and comments.

\section{The Tracial Quasi-Rokhlin Property}\label{SectionTQRP}

\indent

The following definition is based on Definition 1.1 of \cite{OPTrp} 
and also on the behavior of automorphisms induced by minimal 
homeomorphisms.

\begin{dfn}\label{TQRP}
Let $A$ be a separable, unital $C^{*}$-algebra, and let
$\al \in \Aut (A)$. We say that $\al$ has the {\emph{tracial
quasi-Rokhlin property}} if for every $\eps > 0$, every finite set
$F \subset A$, every $n \in \N$, and every positive element $x \in
A$ with $\norm{x} = 1$, there exist $c_{0}, \ldots , c_{n} \in A$
such that:
\begin{enumerate}
\item $0 \leq c_{j} \leq 1$ for $0 \leq j \leq n$;
\item $c_{j} c_{k} = 0$ for $0 \leq j,k \leq n$ and $j \neq k$;
\item $\norm{\al (c_{j}) - c_{j+1}} < \eps$ for $0 \leq j \leq n
- 1$;
\item $\norm{c_{j} a - a c_{j}} < \eps$ for $0 \leq j \leq n$ and
for all $a \in F$;
\item with $c = \sum_{j=0}^{n} c_{j}$, there exist $N \in \N$,
positive elements $e_{0},\ldots, e_{N} \in A$, unitaries $w_{0},
\ldots , w_{N} \in A$, and $d(0), \ldots, d(N) \in \Z$ such that:
\begin{enumerate}
\item $1 - c \leq \sum_{j=0}^{N} e_{j}$;
\item $w_{j} \al^{d(j)} (e_{j}) w_{j}^{*} w_{k} \al^{d(k)} (e_{k})
w_{k}^{*} = 0$ for $0 \leq j,k \leq N$ and $j \neq k$;
\item $w_{j} \al^{d(j)}(e_{j}) w_{j}^{*} \in \overline{xAx}$ for $0
\leq j \leq N$;
\end{enumerate}
\item with $c$ as above, $\norm{cxc} > 1 - \eps$.
\end{enumerate}
\end{dfn}

The key differences between this definition and Definition 1.1 of
\cite{OPTrp} are the change from projections to positive contractions, 
and the statement of condition $(5)$ (as compared to condition $(3)$ 
in Definition 1.1 of \cite{OPTrp}). We also make no assumptions about 
the simplicity of the algebra $A$, but it should be noted that this 
definition is only formulated for cases where the algebra $A$ is 
expected to be {emph{$\al$-simple}} (have no non-trivial 
$\al$-invariant ideals), and this will be assumed in the applications 
that follow.

\begin{lem}\label{ApproxOrthoElts}
Let $A$ be a separable, unital $C^{*}$-algebra, let $\al \in \Aut
(A)$, and let $u$ be the canonical unitary of the crossed product
$C^{*}$-algebra $C^{*}(\Z,A,\al)$. Given any $\eps > 0$ and $n \in
\N$, let $c_{0}, \ldots, c_{n} \in A$ satisfy:
\begin{enumerate}
\item $0 \leq c_{j} \leq 1$ for $0 \leq j \leq n$;
\item $c_{j} c_{k} = 0$ for $0 \leq j,k \leq n$ and $j \neq k$;
\item $\norm{ \al (c_{j}) - c_{j+1} } < \eps$ for $0 \leq j
\leq n - 1$.
\end{enumerate}
Then for $0 \leq j \leq n$ and $1 \leq k \leq n$, we have $\norm{
c_{j} u^{-k} c_{j} } < 3n \eps$ and $\norm{ c_{j} u^{k} c_{j} } <
3n \eps$.
\end{lem}

\begin{proof}
For $0 \leq i \leq j-1$, we have 
\[
\norm{ \al^{k+i} (c_{j-i}) \al^{i} (c_{j-i}) } \leq 2 \norm{ c_{j-i} - \al 
(c_{j-i-1}) } + \norm{\al^{k+i+1} (c_{j-i-1}) \al^{i+1} (c_{j-i-1}) }. 
\]
Also, for $k \neq 0$ we observe that 
\[
\norm{ \al^{k} (c_{0}) - c_{k} } \leq \sum_{i=0}^{k-1} \norm{ 
\al^{k-i} (c_{i}) - \al^{k-i-1} (c_{i+1}) } = \sum_{i=0}^{k-1} 
\norm{ \al (c_{i}) - c_{i+1} } < n \eps
\]
For $1 \leq k \leq n$, combining this with repeated application 
of the previous inequality and using $c_{k}c_{0} = 0$ gives
\begin{align*}
\norm{ c_{j} u^{-k} c_{j} } \leq
\norm{ \al^{k} (c_{j}) c_{j} } &\leq \norm{ \al^{k+j} (c_{0})
\al^{j} (c_{0}) } + 2 \sum_{i=0}^{j-1} \norm{ c_{j-i} - \al
(c_{j-i-1}) } \\
&< \norm{ \al^{k} (c_{0}) c_{0} } + 2 n \eps \\
&\leq \norm{ \al^{k} (c_{0}) - c_{k} } + 2 n \eps \\
&< 3 n \eps.
\end{align*}
The inequality $\norm{ c_{j} u^{k} c_{j} } < 3 n \eps$ is proven 
similarly.
\end{proof}

\begin{lem}\label{NormalizedApprox}
Let $A$ be a separable, unital $C^{*}$-algebra, let $\al \in \Aut 
(A)$, and let $a \in C^{*}(\Z,A,\al)$ be positive and non-zero. 
Then for any $\eps > 0$, there exist $N \in \N$ and $a_{j} \in A$ 
for $- N \leq j \leq N$ such that $\norm{a_{0}} = 1$ and 
\[
\bigg\lVert a - \sum_{j = -N}^{N} a_{j} u^{j} \bigg\rVert < \eps.
\]
\end{lem}

\begin{proof}
Let $E \colon C^{*}(\Z,A,\al) \to A$ be the standard faithful 
conditional expectation. Set $b = a^{1/2}$, which is positive 
and non-zero. Then as $E$ is faithful, it follows that 
\[
E(a) = E(b^{2}) = E(b^{*}b) \neq 0,
\]
By replacing $a$ with $\norm{E(a)}^{-1} a$ if necessary, we 
may assume that $\norm{E(a)} = 1$. Since $C_{c}(\Z,A,\al)$ is dense 
in $C^{*}(\Z,A,\al)$, there exist $N \in \N$ and $\widetilde{b}_{j} 
\in A$ for $-N \leq j \leq N$ such that 
\[
\bigg\lVert \left( a - E(a) \right) - \sum_{j = -N}^{N} \widetilde{b}_{j} 
u^{j} \bigg\rVert < \ts{\frac{1}{2}} \eps.
\]
Using $E( a - E(a) ) = 0$ and $E \left( \sum \widetilde{b}_{j} u^{j}  
\right) = E( \widetilde{b}_{0} )$, we estimate
\[
\big \lVert \widetilde{b}_{0} \big\rVert = \norm{ E \left(a - E(a) 
\right) - E \left( \sum \widetilde{b}_{j} u^{j} \right) } \leq \norm{ 
\left( a - E(a) \right) - \sum \widetilde{b}_{j} u^{j} } < 
\ts{\frac{1}{2}} \eps.
\]
Now set $b_{0} = 0$ and $b_{j} = \widetilde{b}_{j}$ for $1 \leq 
\abs{j} \leq N$. Then 
\[
\bigg\lVert \left( a - E(a) \right) - \sum_{j = -N}^{N} b_{j} u^{j} 
\bigg\rVert \leq \big\lVert \widetilde{b}_{0} \big\rVert + \bigg\lVert 
\left( a - E(a) \right) - \sum_{j = -N}^{N} \widetilde{b}_{j} u^{j} 
\bigg\rVert < \eps.
\]
By defining $a_{0} = E(a)$ and $a_{j} = b_{j}$ for $1 \leq \abs{j} 
\leq N$, it follows that $\norm{a_{0}} = 1$ and 
\[
\bigg\lVert a - \sum_{j = -N}^{N} a_{j} u^{j} \bigg\rVert = 
\bigg\lVert \left( a - E(a) \right) - \sum_{j = -N}^{N} b_{j} u^{j} 
\bigg\rVert < \eps,
\]
as required.
\end{proof}

\begin{thm}\label{SimpleCrossProd}
Let $A$ be a separable, unital $C^{*}$-algebra, let $\al \in 
\Aut (A)$ have the tracial quasi-Rokhlin property, and suppose 
that $A$ is $\al$-simple. Then $C^{*}(\Z,A,\al)$ is simple.
\end{thm}

\begin{proof}
Let $J \subset C^{*}(\Z,A,\al)$ be a non-zero ideal, let $u \in
C^{*}(\Z,A,\al)$ be the canonical unitary in the crossed product,
let $0< \eps < $, and let $a \in J$ be non-zero and positive. By 
Lemma \ref{NormalizedApprox} there exist $n \in \N$ and 
$a_{k} \in A$ for $-n \leq k \leq n$ such that $\norm{a_{0}} = 1$ 
and
\[
\bigg\lVert a - \sum_{k = -n}^{n} a_{k} u^{k} \bigg\rVert < 
\ts{\frac{1}{4}} \eps.
\]
Define continuous functions $f,g \colon [0,1] \to [0,1]$ by
\[
f(t) = \begin{cases} 0 & t \leq 1 - \frac{\eps}{8} \\ 
\frac{16}{\eps} (t-1) + 2 & 1 - \frac{\eps}{8} < t < 1 - 
\frac{\eps}{16} \\ 1 & t \geq 1 - \frac{\eps}{16} \end{cases}
\]
and
\[
g(t) = \begin{cases} 0 & t < 1 - \frac{\eps}{16} \\
\frac{16}{\eps}(t-1) + 1 & t \geq 1 - \frac{\eps}{16}. \end{cases}
\]
Setting $q = g(a_{0}^{1/2})$ and $r = f(a_{0}^{1/2})$, we have 
the relations $q,r \geq 0$, $rq = q$, and $\norm{q} = \norm{r} 
= 1$. Now set $M = \sum_{k \neq 0} \norm{ a_{k} }$ and 
\[
\eps' = \frac{\eps}{12(M(n+1)^{2}+1)} 
\]
and $F = \set{a_{k} \colon -n \leq k \leq n}$. Apply the tracial 
quasi-Rokhlin property with $F,\eps',n$, and $q$ to obtain 
$c_{0},\ldots,c_{n} \in A$ such that
\begin{enumerate}
\item $0 \leq c_{j} \leq 1$ for $0 \leq j \leq n$;
\item $c_{j} c_{k} = 0$ for $0 \leq j,k \leq n$ and $j \neq k$;
\item $\norm{ \al (c_{j}) - c_{j+1} } < \eps'$ for $0 \leq j \leq
n-1$;
\item $\norm{ c_{j} a_{k} - a_{k} c_{j} } < \eps'$ for $0 \leq j
\leq n$ and $-n \leq k \leq n$;
\item with $c = \sum_{j = 0}^{n} c_{j}$, we have $\norm{cqc} 
> 1 - \eps'$.
\end{enumerate}
Using the mutual orthogonality of the $c_{j}$, we have
\begin{align*}
\bigg\lVert \sum_{j=0}^{n} c_{j} a c_{j} - \sum_{j=0}^{n} 
\sum_{k=-n}^{n}c_{j} a_{k} u^{k} c_{j} \bigg\rVert 
&= \bigg\lVert \sum_{j=0}^{n} c_{j} \Big( a -\sum_{k=-n}^{n} a_{k} 
u^{k} \Big) c_{j} \bigg\rVert \\
&\leq \max_{0 \leq j \leq n} \bigg\lVert c_{j} \Big( a - \sum_{k=-n}^{n}
a_{k} u^{k} \Big) c_{j} \bigg\rVert \\
&\leq \bigg\lVert a - \sum_{k=-n}^{n} a_{k} u^{k} \bigg\rVert < \ts{\frac{1}{4}} \eps.
\end{align*}
Since the $c_{j}$ approximately commute with the $a_{k}$, we obtain
\begin{align*}
\bigg\lVert \sum_{j=0}^{n} \sum_{k=-n}^{n} c_{j} a_{k} u^{k} c_{j} -
\sum_{j=0}^{n} \sum_{k=-n}^{n} a_{k} c_{j} u^{k} c_{j} \bigg\rVert 
&= \bigg\lVert \sum_{j=0}^{n} \sum_{k=-n}^{n} ( c_{j} a_{k} - a_{k} 
c_{j} )u^{k} c_{j} \bigg\rVert \\
&\leq \sum_{j=0}^{n} \sum_{k=-n}^{n} \norm{ c_{j} a_{k} - a_{k}
c_{j} } \\
&< 2 (n+1)^{2} \eps' < \ts{\frac{1}{4}} \eps.
\end{align*}
Next, applying Lemma \ref{ApproxOrthoElts} gives
\begin{align*}
\bigg\lVert \sum_{j=0}^{n} \sum_{k=-n}^{n} a_{k} c_{j} u^{k} c_{j} -
\sum_{j=0}^{n} a_{0} c_{j}^{2} \bigg\rVert 
&= \bigg\lVert \sum_{j=0}^{n} \sum_{k \neq 0} a_{k} c_{j} u^{k} 
c_{j}^{2} \bigg\rVert \\
&\leq \sum_{j=0}^{n} \sum_{k \neq 0} \norm{ a_{k} } \norm{ c_{j}
u^{k} c_{j} } \\
&< 3 n (n+1) M \eps' < \ts{\frac{1}{4}} \eps.
\end{align*}
Finally, orthogonality of the $c_{j}$ gives $c^{2} =
\sum_{j=0}^{n} c_{j}^{2}$, and using this we obtain the estimate
\[
\bigg\lVert \sum_{j=0}^{n} a_{0} c_{j}^{2} - c a_{0} c \bigg\rVert 
= \bigg\lVert \sum_{j=0}^{n} (a_{0} c_{j} - c_{j} a_{0} ) c_{j} 
\bigg\rVert \\
\leq \sum_{j=0}^{n} \norm{ a_{0} c_{j} - c_{j} a_{0} } \\
< (n+1) \eps' < \ts{\frac{1}{4}} \eps.
\]
Setting $x = \sum_{j=0}^{n} c_{j} a c_{j}$, it follows that
\[
\norm{ x - c a_{0} c } < \ts{\frac{1}{4}} \eps + \ts{\frac{1}{4}} 
\eps + \ts{\frac{1}{4}} \eps + \ts{\frac{1}{4}} \eps = \eps.
\]

We next show that $\norm{ c a_{0} c }$ is sufficiently large.
With $f(t)$ as before, we compute
\[
\big\lVert a_{0}^{1/2}r - r \big\rVert = \sup_{t \in [0,1]} 
\abs{t f(t) - f(t)} \leq \ts{\frac{1}{8}} \eps
\]
Since $rq = q$ and $\norm{q} = 1$, it follows that $\big\lVert 
a_{0}^{1/2} q - q \big\rVert < \ts{\frac{1}{8}} \eps $. This gives
\begin{align*}
1 - \tfrac{1}{12}\eps < 1 - \eps' < \norm{ cqc } &\leq \big\lVert 
cqc - ca_{0}^{1/2} qc \big\rVert + \big\lVert c a_{0}^{1/2} qc 
\big\rVert \\
&\leq \big\lVert q - a_{0}^{1/2} q \big\rVert + \big\lVert c a_{0}^{1/2} 
\big\rVert \\
&< \tfrac{1}{8}\eps + \norm{ c a_{0}^{1/2} },
\end{align*}
and so $\big\lVert c a_{0}^{1/2} \big\rVert > 1 - \tfrac{5}{24} \eps$. Now 
the assumption $\eps < 1$ gives 
\[
\norm{ c a_{0} c } = \big\lVert (c a_{0}^{1/2} ) ( c a_{0}^{1/2} )^{*} 
\big\rVert = \big\lVert c a_{0}^{1/2} \big\rVert^{2} > ( 1 - \tfrac{5}{24} 
\eps )^{2} = \left( 1 -\ts{\frac{5}{24}} \right)^{2} = \ts{\frac{361}{576}}.
\]
Now suppose that $J \cap A = 0$. By Theorem 3.1.7 of \cite{Mu}, $A + 
J$ is a $C^{*}$-subalgebra of $C^{*}(\Z,A,\al)$, and the assumption 
that $J \cap A = 0$ implies that the projection map $\pi \colon A + 
J \to (A + J)/J$ is isometric when restricted to $A$ (and of course 
it is norm-reducing in general). Since $c a_{0} c \in A$ and $x \in 
J$, it follows that
\[
\ts{\frac{361}{576}} < \norm{ c a_{0} c } = \norm{ \pi ( c a_{0} c ) } =
\norm{ \pi ( c a_{0} c - x ) } \leq \norm{ c a_{0} c - x } <
\ts{\frac{1}{8}},
\]
a contradiction. So there must be a non-zero element in $J \cap A$.
Finally, we claim that $J \cap A$ is an $\al$-invariant ideal of $A$.
To see this, let $b \in J \cap A$. Then $\al (b) = u b u^{*} \in J$
since $J$ is an ideal, and clearly $\al (b) \in A$, so $\al (b) \in
J \cap A$. Thus, $J \cap A$ is a non-zero $\al$-invariant ideal of
$A$, which implies that $J \cap A = A$. It follows that $J =
C^{*}(\Z,A,\al)$, and so $C^{*}(\Z,A,\al)$ is simple.
\end{proof}

\begin{lem}\label{FuncCalc}

Let $f \in C([0,1])$.

\begin{enumerate}

\item\label{FuncCalcProd} For any $\eps > 0$, there 
is a $\dt > 0$ (depending on both $\eps$ and $f$) such that if $A$ 
is a unital $C^{*}$-algebra and $a,b \in A$ satisfy $0 \leq a,b 
\leq 1$, then $\norm{ ab - ba } < \dt$ implies $\norm{ f(b)a - 
af(b) } < \eps$.

\item\label{FuncCalcDiff} For every $\eps > 0$, there 
is a $\dt > 0$ (depending on both $\eps$ and $f$) such that if $A$ 
is a unital $C^{*}$-algebra and $a,b \in A$ satisfy $0 \leq a,b, 
\leq 1$, then $\norm{ a - b } < \dt$ implies $\norm{ f(a) - f(b) } 
< \eps$.

\end{enumerate}

\end{lem}

\begin{proof}
The proofs the the same as in Lemma 2.5.11 of \cite{HLnBook}.
\end{proof}

\begin{lem}\label{AlphaInvIdeal}
Let $A$ be a separable, unital $C^{*}$-algebra, let $\al \in \Aut
(A)$, let $T_{\al}(A)$ denote the space of all $\al$-invariant 
tracial states on $A$, and let $\tau \in T_{\al}(A)$. Then the set 
$I = \set{a \in A \colon \tau (a^{*}a) = 0}$ is an $\al$-invariant 
ideal of $A$.
\end{lem}

\begin{proof}
The map $a \mapsto \tau (a^{*}a)$ is clearly a bounded linear 
functional $A \to \C$, so the set $I = \set{a \in A \colon \tau 
(a^{*}a) = 0}$ is closed. In Section 3.4 of \cite{Mu} it is shown 
that $I$ is a closed left ideal of $A$ (using Theorem 3.3.7 there). 
As $\tau (aa^{*} = \tau (a^{*}a)$, it is clear that $a \in I$ if and 
only if $a^{*} \in I$. Therefore $I$ is a closed left ideal of $A$ 
that is closed under adjoints. But then for any $b \in A$ and $a \in 
I$, we have $b^{*} \in A$ and $a^{*} \in I$. Since $I$ is a left 
ideal of $A$, we get $b^{*} a^{*} \in I$, and since $I$ is closed 
under adjoints, it follows that $ab = (b^{*}a^{*})^{*} \in I$. 
Therefore, $I$ is an ideal of $A$. Finally, given $a \in I$, the 
$\al$-invariance of $\tau$ implies that
\[
\tau ( (\al (a))^{*} (\al (a)) ) = \tau (\al (a^{*}) \al (a) ) =
\tau ( \al (a^{*}a) ) = \tau (a^{*}a) = 0,
\]
and this gives $\al (a) \in I$. Therefore, $I$ is $\al$-invariant.
\end{proof}

\begin{prp}\label{SmallOpenSet}
Let $A$ be a separable, unital $C^{*}$-algebra, let $\al \in \Aut
(A)$, and assume that $A$ is $\al$-simple. Then given any $\tau 
\in T_{\al}(A)$ and any $y \in A$ with $\spec (y) = [0,1]$, and with 
$\mu$ the spectral measure for $\tau$ on  $C^{*}(y,1)$, there is 
an open interval $U \subset [0,1]$ such that $U \neq \varnothing$ 
and $\mu (U) < \eps$.
\end{prp}

\begin{proof}
Since $A$ has no non-trivial $\al$-invariant ideals, Lemma
\ref{AlphaInvIdeal} implies that $\tau (a^{*}a)= 0$ if and only if 
$a = 0$, and so $\tau$ is faithful. Let $V \subset [0,1]$ be any 
non-empty open interval, let $x_{0} \in V$, and choose an $f \in 
C^{*}(y,1) \cong C([0,1])$ such that $f(x_{0}) = 1$ and $\supp (f) 
\subset V$. Then
\[
\mu (V) \geq \int_{0}^{1} f \;d\mu = \tau (f) > 0.
\]
Hence all non-empty open intervals in $[0,1]$ have positive
$\mu$-measure. For $n = 2,3,4,\ldots$ define open intervals 
$U_{n} \subset [0,1]$ by $U_{n} = \big( \frac{1}{n+1}, \frac{1}{n} 
\big)$. Then the collection $\left( U_{n} \right)_{n=1}^{\infty}$ is 
pairwise disjoint, and $\mu (U_{n}) > 0$ for all $n \geq 1$ by the 
previous argument. By pairwise disjointness it follows that 
\[
\sum_{n=2}^{\infty} \mu (U_{n}) = \mu \bigg( \bigcup_{n=2}^{\infty} 
U_{n} \bigg) \leq \mu ([0,1]) = 1
\]
and so this series converges. Thus for some $N \in \N$ we must
have $\sum_{n = N}^{\infty} \mu (U_{n}) < \eps$, and so by setting
$U = U_{N}$ we obtain a non-empty open interval $U \subset 
[0,1]$ with $\mu (U) < \eps$.
\end{proof}

In order for the previous lemma to be useful we must know that our
$C^{*}$-algebra $A$ contains a positive element with spectrum equal
to $[0,1]$. We thus introduce the following definition.

\begin{dfn}\label{ScatteredAlgebra}
A $C^{*}$-algebra $A$ is called {\emph{scattered}} if every state on
$A$ is atomic; that is, given any state $\om$ on $A$, there exist 
pure states $(\om_{j})_{j=1}^{\infty}$ and real numbers 
$(t_{j})_{j=1}^{\infty}$, satisfying $t_{j} \geq 0$ for all $j \geq 
1$ and $\sum_{j=1}^{\infty} t_{j} = 1$, such that $\om = 
\sum_{j=1}^{\infty} t_{j} \om_{j}.$
\end{dfn}

By Theorem 2.2 of \cite{La}, a $C^{*}$-algebra is scattered if and 
only if the spectrum of every self-adjoint element of $A$ is 
countable. The argument in the fourth fact about scattered 
$C^{*}$-algebras on page 61 of \cite{AkeSc} shows that if $A$ is 
unital and not scattered, then there is a positive element $y \in A$ 
with $\spec (y) = [0,1]$. For the case in which we have the most 
interest the algebras involved are not scattered. (See Proposition 
\ref{NonTopScatSpaces} for the justification of this claim.)

\begin{prp}\label{SmallLeftoverTrace}
Let $A$ be a separable, unital $C^{*}$-algebra that is not 
scattered, let $\al \in \Aut (A)$ have the tracial quasi-Rokhlin 
property, and assume that $A$ is $\al$-simple. Then for every 
$\eps > 0$, every finite set $F \subset A$, every $n \in \N$, and 
every $\tau \in T_{\al}(A)$, there exist $c_{0},\ldots,c_{n} \in A$ 
such that
\begin{enumerate}
\item $0 \leq c_{j} \leq 1$ for $0 \leq j \leq n$;
\item $c_{j} c_{k} = 0$ for $0 \leq j,k \leq n$ and $j \neq k$;
\item $\norm{ \al (c_{j}) - c_{j+1} } < \eps$ for $0 \leq j \leq
n-1$;
\item $\norm{ a c_{j} - c_{j} a } < \eps$ for $0 \leq j \leq n$
and for all $a \in F$;
\item with $c = \sum_{j=0}^{n} c_{j}$, we have $\tau (1 - c) 
< \eps$.
\end{enumerate}
\end{prp}

\begin{proof}
Let $\eps > 0$, $F \subset A$ finite, $n \in \N$, and $\tau \in
T_{\al}(A)$ be given. Since $A$ is not scattered, there is a
$y \in A$ with $\spec (y) = [0,1]$. Let $\mu$ be the spectral
measure for $\tau$ on $C^{*}(y,1) \cong C([0,1])$, so that
\[
\tau (f(y)) = \int_{0}^{1} f \; d\mu
\]
for all $f \in C([0,1])$. By Proposition \ref{SmallOpenSet}, there
is a non-empty open interval $I \subset [0,1]$ such that $\mu
(I) < \eps$. Since $I$ is an open interval, there exist $0 <
t_{0} < t_{1} < t_{2} < t_{3} < t_{4} < t_{5} < t_{6} < 1$ such
that $I = (t_{0},t_{6})$. Define continuous functions $f,g \colon
[0,1] \to [0,1]$ by
\[
f(t) = \begin{cases} 0 & 0 \leq t < t_{1} \\ \frac{t -
t_{1}}{t_{2} - t_{1}} & t_{1} \leq t < t_{2} \\ 1 & t_{2} \leq t
< t_{4} \\ \frac{t_{4} - t}{t_{5} - t_{4}} & t_{4} \leq t < t_{5} \\
0 & t_{5} \leq t \leq 1 \end{cases}
\hspace{0.5 in} \mbox{and} \hspace{0.5 in}
g(t) = \begin{cases} 0 & 0 \leq t < t_{2} \\ \frac{t -
t_{2}}{t_{3} - t_{2}} & t_{2} \leq t < t_{3} \\ \frac{t_{3} -
t}{t_{4} - t_{3}} & t_{3} \leq t < t_{4} \\ 0 &
t_{4} \leq t \leq 1 \end{cases}
\]
Then $\supp (f) , \supp (g) \subset I$, $fg = g$, and $f,g \neq 0$.
Set $x = g(y)$ and $b = f(y)$. Then $0 \leq x \leq b \leq 1$ and
$xb = bx = x$. Now for any $a \in \overline{xAx}$ with $0 \leq a
\leq 1$, we have $a = b^{1/2} a b^{1/2} \leq b^{1/2} ( \norm{a}
\cdot 1 ) b^{1/2} \leq b$, and so $\tau (a) \leq \tau (b)$. It
follows that for any $a \in \overline{xAx}$, we have
\[
\tau (a) \leq \tau (b) = \int_{0}^{1} f \; d\mu \leq \mu (I) < \eps.
\]
Now apply the tracial quasi-Rokhlin property with $\eps, F, n$,
and $x$, obtaining $c_{0},\ldots,c_{n} \in A$ such that:
\begin{enumerate}
\item $0 \leq c_{j} \leq 1$ for $0 \leq j \leq n$;
\item $c_{j} c_{k} = 0$ for $0 \leq j,k \leq n$ and $j \neq k$;
\item $\norm{ \al (c_{j}) - c_{j+1} } < \eps$ for $0 \leq j \leq
n-1$;
\item $\norm{ a c_{j} - c_{j} a } < \eps$ for $0 \leq j \leq n$
and for all $a \in F$;
\item with $c = \sum_{j=0}^{n} c_{j}$, there exists $N \in \N$,
positive elements $e_{0},\ldots, e_{N} \in A$, unitaries $w_{0},
\ldots , w_{N} \in A$, and $d(0), \ldots, d(N) \in \Z$ such that: 
\begin{enumerate}
\item $1 - c \leq \sum_{j=0}^{N} e_{j}$; 
\item $\al^{d(j)} (e_{j}) \al^{d(k)} (e_{k}) = 0$ for $0 \leq j,k 
\leq N$;
\item $j \neq k$, and $w_{j} \al^{d(j)}(e_{j}) w_{j}^{*} \in
\overline{xAx}$ for $0 \leq j \leq N$.
\end{enumerate}
\end{enumerate}
Since each $w_{j} \al^{d(j)} (e_{j}) w_{j}^{*} \in \overline{xAx}$, 
it follows that $\sum_{j=0}^{N} w_{j} \al^{d(j)} (e_{j}) w_{j}^{*} 
\in \overline{xAx}$, and so
\[
\tau \bigg( \sum_{j=0}^{N} w_{j} \al^{d(j)} (e_{j}) w_{j}^{*}
\bigg) < \eps
\]
Then the linearity and $\al$-invariance of $\tau$ imply that
\[
\tau (1-c) \leq \sum_{j=0}^{N} \tau (e_{j}) 
= \sum_{j=0}^{N} \tau \left( \al^{d(j)} (e_{j}) \right)  
= \sum_{j=0}^{N} \tau \left( w_{j} \al^{d(j)} (e_{j}) w_{j}^{*}
\right) < \eps,
\]
which completes the proof.
\end{proof}

\begin{thm}\label{TracialStateSpace}
Let $A$ be a separable, unital $C^{*}$-algebra that is not
scattered, let $\al \in \Aut (A)$ have the tracial quasi-Rokhlin 
property, and suppose that $A$ is $\al$-simple. Then the 
restriction map $T( C^{*}(\Z,A,\al) ) \to T_{\al}(A)$ is bijective.
\end{thm}

\begin{proof}

We first verify that every trace on $T ( C^{*} (\Z,A,\al) )$ is 
$\al$-invariant when restricted to $A$, so that the restriction 
map indeed has codomain $T_{\al}(A)$. For any $\tau \in T ( 
C^{*}(\Z,A,\al) )$ and any $a \in A$, we have
\[
\tau (\al (a)) = \tau (uau^{*}) = \tau (au^{*}u) = \tau (a),
\]
and so this is in fact the case.

Next, we show that the restriction map is injective. Let $\tau \in 
T(C^{*}(\Z,A,\al) )$, let $\eps > 0$ be given, let $a \in A$ be 
non-zero, let $k \in \N \setminus \set{0}$, and let $u \in 
C^{*}(\Z,A,\al)$ be the canonical unitary. Set $F = \set{a}$ and 
choose $n \in \N$ such that $n > k$ and
\[
\frac{1}{n} < \frac{ \eps^{2} }{16 k^{2} ( \norm{a^{*}a}
+ 1) }.
\]
Apply Lemma \ref{FuncCalc}(\ref{FuncCalcProd}) with $f(x) = 
\sqrt{x}$ to obtain $\dt_{1}(\eps) > 0$ such that for all $b,e \in 
A$ with $0 \leq b,e \leq 1$ and $\norm{ be - eb } < \dt_{1}(\eps)$, 
we have
\[
\big\lVert b^{1/2}e - eb^{1/2} \big\rVert < \frac{\eps}{8n}.
\]
Similarly, apply Lemma \ref{FuncCalc}(\ref{FuncCalcDiff}) with 
the same $f$ to obtain $\dt_{2}(\eps) > 0$ such that for all $b,e 
\in A$ with $0 \leq b,e \leq 1$ and $\norm{ e - b } < \dt_{2}(\eps)$, 
we have
\[
\big\lVert e^{1/2} - b^{1/2} \big\rVert < \frac{\eps}{8nk (\norm{a} +
1) }.
\]
Define
\[
\dt = \min \set{ \frac{1}{2n^{3} + n^{2} + 1} , \dt_{1}
(\eps) , \dt_{2} (\eps) , \frac{\eps^{2}}{4 (\tau (a^{*}a)
+ 1 )} } > 0
\]
and apply Proposition \ref{SmallLeftoverTrace} with $\dt, F, n$, 
and $\tau$ (identifying $\tau$ with its image in $T_{\al}(A)$ 
under the restriction map) to obtain $c_{0},\ldots,c_{n} \in A$ 
such that:
\begin{enumerate}
\item $0 \leq c_{j} \leq 1$ for $0 \leq j \leq n$;
\item $c_{j} c_{k} = 0$ for $0 \leq j,k \leq n$ and $j \neq k$;
\item $\norm{ \al (c_{j}) - c_{j+1} } < \dt$ for $0 \leq j \leq
n-1$;
\item $\norm{ c_{j}a - ac_{j} } < \dt$ for $0 \leq j \leq n$;
\item with $c = \sum_{j=0}^{n} c_{j}$, we have $\tau (1 - c)
< \dt$.
\end{enumerate}

By the choice of $\dt$, and since automorphisms commute with
continuous functional calculus, we further obtain
\[
\big\lVert \al (c_{j}^{1/2}) - c_{j+1}^{1/2} \big\rVert < \frac{\eps}{
8nk (\norm{a} + 1) }
\]
for $0 \leq j \leq n - k$, and
\[
\big\lVert c_{j}^{1/2} a - a c_{j}^{1/2} \big\rVert < \frac{\eps}{8n}
\]
for $0 \leq j \leq n$. It is easy to see that $0 \leq c \leq 1$
and hence also $0 \leq 1 - c \leq 1$. Then $(1 - c)^{1/2}$ is
a well-defined positive element of $A$ that satisfies $1 - c \leq 
1$. Observing that that continuous functions $f_{0}, f_{1} \colon 
[0,1] \to [0,1]$ given by $f_{0}(t) = t^{2}$ and $f_{1}(t) = t$ 
satisfy $f_{0} \leq f_{1}$, continuous functional calculus gives 
$(1 - c)^{2} \leq (1 - c)$. It follows that $\tau ( (1-c)^{2} ) \leq 
\tau (1-c)$ and so the Cauchy-Schwarz inequality yields
\begin{align*}
\abs{ \tau (au^{k}(1-c)) }^{2} &\leq \tau ( (1-c)^{*}(1-c) )
\tau ((au^{k})(au^{k})^{*}) \\
&= \tau ((1-c)^{2}) \tau ( (au^{k})^{*}(au^{k}) ) \\
&= \tau ((1-c)^{2}) \tau (u^{-k}a^{*}au^{k}) \\
&= \tau ((1-c)^{2}) \tau (a^{*}a) \\
&\leq \tau (1-c) \tau (a^{*}a) \\
&< \dt \tau (a^{*}a).
\end{align*}
Hence $\abs{ \tau (au^{k}(1-c)) } < \sqrt{ \dt \tau (a^{*}a) } < 
\ts{\frac{1}{2}} \eps$. 

Next, let $e,b \in A$ be positive and orthogonal. We compute
\[
\big\lVert b^{1/2}e \big\rVert^{2} = \big\lVert 
(b^{1/2}e)^{*}(b^{1/2}e) \big\rVert = \norm{ ebe } = 0,
\]
which implies that $b^{1/2}e = 0$. This gives
\[
\big\lVert e^{1/2} b^{1/2} \big\rVert^{2} = \big\lVert 
(e^{1/2} b^{1/2})^{*} (e^{1/2} b^{1/2}) \big\rVert = 
\big\lVert b^{1/2} e b^{1/2} \big\rVert = 0,
\]
which implies that $e^{1/2} b^{1/2} = 0$ as well. In particular,
for $0 \leq j \leq n - k$, we have $c_{j}^{1/2} c_{j+k}^{1/2} = 0$,
and so $\tau ( c_{j+k}^{1/2}au^{k}c_{j}^{1/2} ) = \tau ( a
u^{k}c_{j}^{1/2}c_{j+k}^{1/2} ) = 0$. For $0 \leq j \leq n - k$,
we also have the inequality
\begin{align*}
\big\lVert \al^{k} (c_{j}^{1/2}) - c_{j+k}^{1/2} \big\rVert &\leq
\sum_{i=0}^{k-1} \big\lVert \al^{k-i} (c_{j+i}^{1/2}) -
\al^{k-i-1} (c_{j+i+1}^{1/2}) \big\rVert \\
&= \sum_{i=0}^{k-1} \big\lVert \al (c_{j+i}^{1/2}) -
c_{j+i+1}^{1/2} \big\rVert < k \dt.
\end{align*}
It follows that for $0 \leq j \leq n - k$,
\begin{align*}
\abs{ \tau (au^{k}c_{j}) } &= \big\lvert \tau (au^{k}c_{j}^{1/2} 
c_{j}^{1/2}) \big\rvert \\
&= \big\lvert \tau (a \al^{k}(c_{j}^{1/2})u^{k}c_{j}^{1/2}) \big\rvert \\
&\leq \big\lvert \tau (a \al^{k} (c_{j}^{1/2}) u^{k}c_{j}^{1/2} ) -
\tau (ac_{j+k}^{1/2} u^{k}c_{j}^{1/2}) \big\rvert + \big\lvert \tau (a
c_{j+k}^{1/2} u^{k} c_{j}^{1/2}) \big\rvert \\
&= \big\lvert \tau (a (\al^{k} (c_{j}^{1/2}) - c_{j+k}^{1/2} ) u^{k}
c_{j}^{1/2}) \big\rvert + \big\lvert \tau ( (a c_{j+k}^{1/2} - 
c_{j+k}^{1/2} a) u^{k} c_{j}^{1/2}) \big\rvert \\
&\leq \norm{ \tau } \big\lvert a ( \al^{k} (c_{j}^{1/2}) -
c_{j+k}^{1/2} ) u^{k} c_{j}^{1/2} \big\rvert + \norm{ \tau } \big\lvert 
(a c_{j+k}^{1/2} - c_{j+k}^{1/2} a) u^{k} c_{j}^{1/2} \big\rvert \\
&\leq \norm{a} \big\lvert \al^{k} (c_{j}^{1/2}) - c_{j+k}^{1/2} 
\big\rvert + \big\lvert a c_{j+k}^{1/2} - c_{j+k}^{1/2} a \big\rvert \\
&< \norm{a} k \Big( \frac{\eps}{8nk (\norm{a} + 1)} \Big)
+ \frac{\eps}{8n} \\ 
&< \frac{\eps}{4n}.
\end{align*}

For $0 \leq k \leq n-1$ the $\al$-invariance of $\tau$ implies that
\[
\abs{ \tau (c_{j+1}) - \tau (c_{j}) } = \abs{ \tau (c_{j+1}) -
\tau (\al (c_{j})) } = \abs{ \tau ( c_{j+1} - \al (c_{j}) ) } \leq
\norm{ c_{j+1} - \al (c_{j}) } < \dt,
\]
and so we obtain
\begin{align*}
\bigg\lvert (n+1) \tau (c_{0}) - \sum_{j=0}^{n} \tau (c_{j}) \bigg\rvert 
\leq \sum_{j=1}^{n} \abs{ \tau (c_{j}) - \tau (c_{0}) } 
&\leq \sum_{j=1}^{n} \sum_{i=0}^{j-1} \abs{ \tau (c_{j-i}) - \tau
(c_{j-i-1}) } \\
&< \sum_{j=1}^{n} j \dt \leq n^{2} \dt.
\end{align*}
Now, since $0 \leq c \leq 1$, we have $\sum_{j=0}^{n} \tau (c_{j})
\leq 1$. Combining this with the previous result gives
\[
(n+1) \tau (c_{0}) < n^{2} \dt + \sum_{j=0}^{n} \tau (c_{j})
\leq n^{2} \dt + 1,
\]
and this implies that
\[
\tau (c_{0}) < \frac{ n^{2} \dt + 1}{n + 1} < \frac{
\tfrac{1}{2n} + 1 }{n + 1} < \frac{1}{n}.
\]
Further, since $\abs{ \tau (c_{j}) - \tau (c_{0}) } < n \dt$ for
$1 \leq j \leq n$ (this follows by iterating one of the previous
inequalities with the triangle inequality), we conclude that for
$0 \leq j \leq n$, we have
\[
\tau (c_{j}) < n \dt + \tau (c_{0}) < n \dt + \frac{ n^{2}
\dt + 1}{n + 1} < \frac{ (2n^{2} + n) \dt + 1}{n + 1} <
\frac{ \tfrac{1}{n} + 1}{n + 1} = \frac{1}{n}.
\]
Now $0 \leq c_{j} \leq 1$ implies that $c_{j}^{2} \leq c_{j}$ by 
the same functional calculus argument that was used to show 
$(1 - c)^{2} \leq 1 - c$, and consequently $0 \leq \tau (c_{j}^{2}) 
\leq \tau (c_{j})$. Applying Theorems 3.3.2 and 3.3.7 of \cite{Mu} 
gives
\begin{align*}
\abs{ \tau (au^{k}c_{j}) }^{2} \leq \norm{ \tau } \tau ( (au^{k}
c_{j})^{*} (au^{k}c_{j}) ) 
&= \tau ( (u^{k}c_{j})^{*} a^{*}a (u^{k}c_{j}) ) \\
&\leq \norm{ a^{*}a } \tau ( (u^{k}c_{j})^{*} (u^{k}c_{j}) ) \\
&= \norm{ a^{*}a } \tau (c_{j}^{2}) 
\leq \norm{ a^{*}a } \tau (c_{j}) 
< \frac{ \norm{ a^{*}a } }{n} 
< \frac{ \eps^{2} }{ 16k^{2} },
\end{align*}
which implies $\abs{ \tau (au^{k}c_{j}) } < \frac{\eps}{4k}$. Finally, 
we compute
\begin{align*}
\abs{ \tau (au^{k}) } &\leq \abs{ \tau (au^{k} (1-c)) } + \abs{
\tau (au^{k}c) } \\
&< \ts{\frac{1}{2}} \eps + \sum_{j=0}^{n-k} \abs{\tau (au^{k}c_{j})} 
+ \sum_{j=n-k+1}^{n} \abs{ \tau (au^{k}c_{j}) } \\
&< \ts{\frac{1}{2}} \eps + \sum_{j=0}^{n-k} \frac{\eps}{4n} +
\sum_{j=n-k+1}^{n} \frac{\eps}{4k} \\
&\leq \ts{\frac{1}{2}} \eps + \ts{\frac{1}{4}} \eps + 
\ts{\frac{1}{4}} \eps \\
&= \eps.
\end{align*}
Since $\eps > 0$ was arbitrary, it follows that $\tau (a u^{k}) = 0$.
Now if $k \in \Z$ with $k < 0$, then the previous argument implies
that $\tau (a^{*}u^{-k}) = 0$, and therefore
\[
\tau (au^{k}) = \tau (u^{k}a) = \tau ((a^{*}u^{-k})^{*}) =
\overline{ \tau (a^{*} u^{-k}) } = 0.
\]
Thus for any $\tau \in T ( C^{*}(\Z,A,\al) )$, any non-zero
$a \in A$, and any $k \in \Z \setminus \set{0}$, we have $\tau
(au^{k}) = 0$. Let $E \colon C^{*}(\Z,A,\al) \to A$ be the standard
conditional expectation. Then for any element $\sum_{j=-N}^{N} 
a_{j} u^{j} \in C_{c}(\Z,A,\al)$, we have
\[
\tau \bigg( \sum_{j=-N}^{N} a_{j} u^{j} \bigg) = \tau (a_{0}) 
= \tau \bigg( E \bigg( \sum_{j=-N}^{N} a_{j} u^{j} \bigg) 
\bigg),
\]
and so $\tau = \tau \circ E$ on a dense subset of $C^{*}(\Z,A,\al)$.
This implies that the restriction map $T ( C^{*}(\Z,A,\al) ) \to
T_{\al}(A)$ is injective.

For surjectivity, let $\tau \in T_{\al}(A)$, and let $E$ be the 
standard conditional expectation introduced above. We claim that 
$\widetilde{\tau} = \tau \circ E$ is a tracial state on 
$C^{*}(\Z,A,\al)$ that satisfies $\widetilde{\tau} \vert_{A} = \tau$. 
It is clear that $\widetilde{\tau}$ is a positive linear map since 
both $\tau$ and $E$ are positive, and we compute 
$\widetilde{\tau}(1) = \tau (E(1)) = \tau (1) = 1$. Let $a = a_{0} 
u^{m}$ and $b = b_{0} u^{n}$ for some $a_{0}, b_{0} \in A$ and $m,n 
\in \Z$. Then we obtain the formulas 
\[
ab = a_{0} u^{m} b_{0} u^{n} = a_{0} \al^{m} (b_{0}) u^{m+n}
\]
and 
\[
ba = b_{0} u^{n} a_{0} u^{m} = b_{0} \al^{n} (a_{0}) u^{m+n}
\]
If $m \neq n$, then $E (ab) = 0 = E (ba)$, and consequently 
$\widetilde{\tau} (ab) = 0 = \widetilde{\tau} (ba)$. So assume that 
$m = -n$, which implies $E (ab) = a_{0} \al^{-n} (b_{0})$ and $E 
(ba) = b_{0} \al^{n} (a_{0})$. Using the $\al$-invariance of $\tau$ 
and the trace property, we obtain  
\[
\tau ( a_{0} \al^{-n} (b_{0}) ) = \tau (\al^{-n} ( \al^{n} (a_{0}) 
b_{0} ) = \tau ( \al^{n} (a_{0}) b_{0} ) = \tau ( b_{0} \al^{n} 
(a_{0}) ),
\]
which implies that 
\[
\widetilde{\tau} (ab) = \tau ( E (ab) ) = \tau ( E (ba) ) = 
\widetilde{\tau} (ba).
\]
Since the dense subset $C_{c}(\Z,A,\al)$ of $C^{*}(\Z,A,\al)$ is 
linearly spanned by elements of the form $a u^{n}$ for $a \in A$ 
and $n \in \Z$, it follows that $\widetilde{\tau}$ is a tracial state on 
$C^{*}(\Z,A,\al)$. Since $E(a) = a$ for all $a \in A$, we clearly 
have $\widetilde{\tau} \vert_{A} = \tau$, which completes the proof 
that the restriction map $T(C^{*}(\Z,A,\al)) \to T_{\al}(A)$ is 
surjective, and hence a bijection.
\end{proof}

\section{Automorphisms of $C(X,A)$ with the Tracial Quasi-Rokhlin 
Property}\label{SectionAutomsWithTQRP}

\indent

Our next goal is to study the automorphisms for a sort of 
noncommutative minimal dynamical system, where the commutative 
$C^{*}$-algebra $C(X)$ is tensored with a simple, separable, unital, 
infinite-dimensional $C^{*}$-algebra $A$. We prove that 
automorphisms of such algebras which take the action of a minimal 
homeomorphism when restricted to the central subalgebra $C(X)$ 
satisfy the tracial quasi-Rokhlin property (under some additional 
technical assumptions). After further consideration of the structure of 
these algebras, it will follow that our results for crossed products by 
automorphisms with the tracial quasi-Rokhlin property in Section 
\ref{SectionTQRP} will apply to their associated crossed product 
$C^{*}$-algebras.

\begin{ntn}\label{CtsFcnsXtoA}
Throughout, we let $X$ be an infinite compact metrizable space, 
and let $h \colon X \to X$ be a minimal homeomorphism. The 
corresponding minimal dynamical system $(X,h)$ will sometimes 
be denoted simply by $X$, with the homeomorphism $h$ 
understood. We denote by $M_{h}(X)$ the space of $h$-invariant 
Borel probability measures on $X$. Whenever necessary, it will be 
assumed that $X$ is a metric space with metric $d$. In this case, 
for $x \in X$ and $\eps > 0$, we will denote the $\eps$-ball 
centered at $x$ by 
\[
B(x,\eps) = \set{y \in X \colon d(x,y) < \eps}.
\]
We denote the boundary of a set $A \subset X$ by $\del A$. In 
particular, if $U \subset X$ is open then $\del U = \overline{U} 
\setminus U$, and if $C \subset X$ is closed then $\del C = C 
\setminus \sint (C)$. We take $A$ to be a simple, unital, separable, 
infinite-dimensional nuclear $C^{*}$-algebra. Form the algebra 
$C(X,A)$, which we frequently identify with $C(X) \ten A$ in the 
canonical way. For $f \in C(X)$ and $a \in A$, we denote  by 
$f \ten a$ the element of $C(X,A)$ given by $(f \ten a) (x) = f(x)a$ 
for all $x \in X$, noting that these elements span $C(X,A)$. We 
identify $C(X)$ with the central subalgebra of $C(X,A)$ given by 
$\set{f \ten 1 \colon f \in C(X)}$.
\end{ntn}

We will eventually want it to be the case that $C(X,A)$ has 
cancellation of projections and order on projections determined 
by traces. This will in fact occur for many reasonable choices of 
$A$. The proof of the following proposition uses heavy machinery, 
and it is likely that its conclusion applies to a more general class of 
algebras $A$.

\begin{prp}\label{CancelandOrderonProjns}
Let $(X,h)$ and $A$ be as in Notation \ref{CtsFcnsXtoA}. Assume 
in addition that $A$ has tracial rank zero and satisfies the Universal 
Coefficient Theorem. Then $C(X,A)$ has cancellation of projections, 
and order on projections over $C(X,A)$ is determined by traces.
\end{prp}

\begin{proof}
Since $A$ has tracial rank zero and satisfies the Universal 
Coefficient Theorem, Lin's classification theory (see \cite{HLnTRZ}) 
implies that $A$ is a simple infinite-dimensional $\mathrm{AH}$-algebra 
with no dimension growth. Write $A \cong \dirlim A_{n}$, where the 
$A_{n}$ are homogeneous algebras and the direct system has no 
dimension growth, and observe that 
\[
C(X,A) \cong C(X) \ten A \cong C(X) \ten \left( \dirlim A_{n} \right) 
\cong \dirlim C(X) \ten A_{n} \cong \dirlim C(X,A_{n}).
\]
Hence $C(X,A)$ itself is a simple, infinite-dimensional inductive 
limit of homogeneous algebras with no dimension growth. Now Corollary 
1.9 of \cite{PhRsha2} implies that the associated direct system has 
strict slow dimension growth. By Theorem 3.7 of \cite{MP}, it follows 
that $C(X,A)$ has cancellation and order on projections over $C(X,A)$ 
is determined by traces.
\end{proof}

Unfortunately, the assumption that $A$ has real rank zero (in addition 
to the other standard hypotheses) is not sufficient to guarantee that 
$C(X,A)$ has cancellation. We are thankful to Efren Ruiz for pointing 
out the following result.

\begin{prp}\label{PurelyInfNoCancel}
Let $(X,h)$ and $A$ be as in Notation \ref{CtsFcnsXtoA}, and assume 
that $A$ is purely infinite. Then $C(X,A)$ does not have cancellation 
of projections.
\end{prp}

\begin{proof}
Let $p \in A$ be a non-zero projection. Since $A$ is purely infinite, 
there is a unital embedding 
\[
\iota \colon \mathcal{O}_{2} \to (1 \ten p) C(X,A) (1 \ten p) \subset 
C(X,A).
\]
Then $0 + \iota (1) \sim \iota(1) + \iota (1)$, but 0 is not Murray-von 
Neumann equivalent to $\iota (1)$.
\end{proof}

\begin{rmk}
It is possible that $A$ being infinite-dimensional, stably finite, and 
having real rank zero is sufficient to imply that $C(X,A)$ has 
cancellation of projections and order on projections determined by 
traces. On the other hand, real rank zero is certainly not necessary. 
If $X$ is connected and $\JS$ is the Jiang-Su algebra, then 
$C(X,\JS)$ has no nontrivial projections and so both properties 
hold.
\end{rmk}

\begin{prp}\label{TracesOfCXA}
Let $(X,h)$ and $A$ be as in Notation \ref{CtsFcnsXtoA}. For $x \in 
X$, denote by $\mu_{x} \in M(X)$ (where $M(X)$ is the space of all 
Borel probability measures on $X$) the point-mass measure 
concentrated at $x$. Then $T(C(X,A))$ is the weak*-closed convex 
hull of the set 
\[
\set{ \mu_{x} \ten \tau \colon x \in X, \tau \in T(A) }.
\]
\end{prp}

\begin{proof}
It suffices to prove that for any $\eps > 0$, any finite set $\mathcal{F} 
\subset C(X,A)$, and any $\tau \in T(C(X,A))$, there exist $x_{1},\ldots, 
x_{n} \in X$, $\tau_{1},\ldots,\tau_{n} \in T(A)$, and $\ld_{1},\ldots,
\ld_{n} \in [0,1]$ with $\sum_{j=1}^{n} \ld_{j} = 1$ such that 
\[
\bigg\lvert \tau (b) - \sum_{j=1}^{n} \ld_{j} \mu_{x_{j}} \ten \tau_{j} (b) 
\bigg\rvert < \eps
\]
for all $b \in \mathcal{F}$. Let $\eps > 0$, a finite set $\mathcal{F} 
\subset C(X,A)$, and $\tau \in T(C(X,A))$ be given. Choose $\dt > 0$ 
such that whenever $E \subset X$ with $\diam (E) < \dt$, we have 
$\norm{b(x) - b(y)} < \eps$ for any $b \in \mathcal{F}$ and all $x,y \in 
E$ (this can be done since $\mathcal{F}$ is finite). Choose an open 
cover $\set{E_{j}}_{j=1}^{n}$ for $X$ such that $\diam (E_{j}) < \dt$ 
for $1 \leq j \leq n$. Choose a partition of unity $\set{g_{j}}_{j=1}^{n}$ 
subordinate to this cover. For $1 \leq j \leq E_{j}$, choose $x_{j} \in 
E_{j}$. By the choice of the sets $E_{j}$ we then have 
\[
\bigg\lVert b - \sum_{j=1}^{n} g_{j} \ten b(x_{j}) \bigg\rVert < \eps
\]
for each $b \in \mathcal{F}$. For $1 \leq j \leq n$ define $\sm_{j}(a) = 
\tau (g_{j} \ten a)$. Then each $\sm_{j}$ is clearly a positive linear 
functional on $A$ that satisfies the trace property, and so there is 
a $\ld_{j} \geq 0$ and a $\tau_{j} \in T(A)$ such that $\sm_{j} = 
\ld_{j} \tau_{j}$. Now for any $b \in \mathcal{F}$, we have 
\begin{align*}
\bigg\lvert \tau(b) - \sum_{j=1}^{n} \ld_{j} \mu_{x_{j}} \ten \tau_{j} (b) 
\bigg\rvert &= \bigg\lvert \tau (b) - \sum_{j=1}^{n} \ld_{j} \tau_{j} 
(b(x_{j})) \bigg\rvert \\ 
&= \bigg\lvert \tau (b) - \sum_{j=1}^{n} \tau (g_{j} \ten b(x_{j}) 
\bigg\rvert 
\leq \bigg\lVert b - \sum_{j=1}^{n} g_{j} \ten b(x_{j}) \bigg\rVert < \eps,
\end{align*}
which completes the proof.
\end{proof}

\begin{lem}\label{AlphaInverse}
Let $(X,h)$ and $A$ be as in Notation \ref{CtsFcnsXtoA}. Let $\al 
\colon X \to \Aut (A)$ (where $\al(x)$ will be denoted $\al_{x}$) be 
a map which is continuous in the strong operator topology. (In other 
words, for each $a \in A$ the mapping $x \to \al_{x}(a)$ is 
norm-continuous.) Then the map $\al^{-1} \colon X \to \Aut (A)$ 
given by $\al^{-1}(x) = \al_{x}^{-1}$ is continuous in the strong 
operator topology.
\end{lem}

\begin{proof}
This is a straightforward calculation.
%
%
%
\end{proof}

\begin{prp}\label{AutomForCXA}
Let $(X,h)$ and $A$ be as in Notation \ref{CtsFcnsXtoA}. 
Let $\al \colon X \to \Aut (A)$ be a map which is continuous in the 
strong operator topology. Define a map $\bt \colon C(X,A) \to 
C(X,A)$ by $\bt (f) (x) = \al_{x}(f \circ h^{-1}(x))$ for each $x 
\in X$. Then $\bt \in \Aut (C(X,A))$.
\end{prp}

\begin{proof}
We first verify that $\bt(f)$ is continuous for $f \in C(X,A)$. 
Let $\eps > 0$ be given, let $f \in C(X,A)$, and let $x \in X$. 
Since $f \circ h^{-1}(x) \in A$ and $\al$ is continuous in the 
strong operator topology,  there exists $\dt_{1} > 0$ such that 
$d(x,y) < \dt_{1}$ implies $\norm{ \al_{x} (f \circ h^{-1} (x)) - 
\al_{y} (f \circ h^{-1} (x)) } < \eps/2$. Since $f$ is continuous, 
there exists $\dt_{2} > 0$ such that $d(x,y) < \eps/2$ implies 
$\norm{ f(x) - f(y) } < \eps/2$. Also, since $h$ is a homeomorphism, 
there is a $\dt_{3} > 0$ such that $d(x,y) < \dt_{3}$ implies 
$d(h^{-1}(x),h^{-1}(y)) < \dt_{2}$. Now let $\dt = \min \set{\dt_{1},
\dt_{2},\dt_{3}}$. Then for all $y \in X$ with $d(x,y) < \dt$, we 
have
\begin{align*}
\norm{ \bt (f)(x) - \bt (f)(y) } &= \norm{ \al_{x} ( f \circ h^{-1}
(x) ) - \al_{y} ( f \circ h^{-1} (y) ) } \\
&< \frac{\eps}{2} + \norm{\al_{x}} \norm{ f \circ h^{-1}(x) - f
\circ h^{-1}(y) } \\
&< \eps.
\end{align*}
Thus $\bt (f)$ is continuous at $x$. Since this holds for any $x
\in X$, it follows that $\bt (f) \in C(X,A)$. 

Since the operations on $C(X,A)$ are given pointwise, each 
$\al_{x}$ is an automorphism on $A$ for $x \in X$, and the map 
$f \mapsto f \circ h^{-1}$ is an automorphism of $C(X)$, it follows 
easily that for all $f,g \in C(X,A)$, we have $\bt(f + g) = \bt (f) + 
\bt (g)$, $\bt (fg) = \bt (f) \bt (g)$, and $\bt (f^{*}) = \bt (f)^{*}$. 
This implies that $\bt$ is a $*$-homomorphism.

Next suppose that $f \in \ker (\bt)$. Then $\bt(f)(x) = 0$ for all
$x \in X$, and so $\al_{x} (f \circ h^{-1}(x)) = 0$ for all $x \in
X$. Since each $\al_{x}$ is an automorphism of $A$, this implies
that $f \circ h^{-1} (x) = 0$ for each $x \in X$, and hence $f \circ
h^{-1} = 0$. As $h$ is a homeomorphism, it follows that $f = 0$. 
Now let $f \in C(X,A)$. Define $g \colon X \to A$ by $g(x) = 
\al_{x}^{-1} (f \circ h(x))$. That $g$ is continuous follows from 
the same argument that shows $\bt$ is continuous, using 
Lemma \ref{AlphaInverse}. Now for each $x \in X$, $\bt (g)(x) 
= \al_{x} ( \al_{x}^{-1} ( (f \circ h) \circ h^{-1} (x) ) ) = f(x)$, and 
so $\bt (g) = f$. It follows that $\bt$ is bijective, and hence $\bt 
\in \Aut (C(X,A))$.
\end{proof}

\begin{prp}\label{PowersOfAlpha}
Let $(X,h)$ and $A$ be as in Notation \ref{CtsFcnsXtoA}. 
Let $\al \colon X \to \Aut (A)$ be continuous in the strong operator
topology. For $k \in \Z \setminus \set{0}$, we define $\al^{(k)}
\colon X \to \Aut (A)$ by $\al^{(k)}(x) = \al_{x} \circ \al_{h^{-1}(x)}
\circ \cdots \circ \al_{h^{-(k-1)}(x)}$ if $k \geq 1$ and
$\al^{(k)}(x) = \al_{h(x)} \circ \cdots \circ \al_{h^{\abs{k}}(x)}$
if $k < 0$, henceforth denoting $\al^{(k)}(x)$ by
$\al_{x}^{(k)}$. Then $\al^{(k)}$ is continuous in the strong
operator topology. Moreover, the map $\al^{-(k)} \colon X \to
\Aut (A)$, defined by $\al_{x}^{-(k)} = \al_{h^{-(k-1)}(x)}^{-1}
\circ \cdots \al_{h^{-1}(x)}^{-1} \circ \al_{x}^{-1}$ for $k \geq 1$
and $\al_{x}^{-(k)} = \al_{h^{\abs{k}}(x)}^{-1} \circ \cdots \circ
\al_{h(x)}^{-1}$ for $k < 0$, is continuous in the strong operator
topology and satisfies $\al_{x}^{-(k)} = (\al_{x}^{(k)})^{-1}$ for
all $x \in X$.
\end{prp}

\begin{proof}
First, assume that $k \geq 1$. We proceed by induction on $k$. 
When $k = 1$ the map $\al^{(1)} \colon X \to \Aut (A)$ is simply
$\al_{x}^{(1)} = \al_{x}$, which is continuous in the strong
operator topology by assumption. Suppose that $\al^{(k)}$ is
continuous in the strong operator topology for some $k \geq 1$. 
Let $\eps > 0$ be given, let $a \in A$, and let $x \in X$. Then 
there is a $\dt_{1} > 0$ such that $d(x,y) < \dt_{1}$ implies 
$\norm{ \al_{x}^{(k)} (a) - \al_{y}^{(k)} (a) } < \ts{\frac{1}{2}} \eps$. 
Further, with $b = \al_{x}^{(k)} (a)$, the strong operator 
continuity of $\al = \al^{(1)}$ gives a $\dt_{2} > 0$ such that 
$d(x,y) < \dt_{2}$ implies $\norm{ \al_{x} (b) - \al_{y} (b) } < 
\ts{\frac{1}{2}} \eps$. Let $\dt = \min \set{\dt_{1},\dt_{2}}$. Then 
$d(x,y) < \dt$ implies that
\begin{align*}
\norm{ \al_{x}^{(k+1)} (a) - \al_{y}^{(k+1)} (a) } &\leq \norm{
\al_{x}^{(k+1)} (a) - \al_{y} \circ \al_{x}^{(k)} (a) } + \norm{
\al_{y} \circ \al_{x}^{(k)} (a) - \al_{y}^{(k+1)} (a) } \\
&= \norm{ \al_{x} ( \al_{x}^{(k)}(a) ) - \al_{y} (\al_{x}^{(k)}
(a) ) } + \norm{ \al_{y} ( \al_{x}^{(k)} (a) - \al_{y}^{(k)} (a)
) } \\
&\leq \norm{ \al_{x} (b) - \al_{y} (b) } + \norm{ \al_{x}^{(k)} (a)
- \al_{y}^{(k)} (a) } \\
&< \ts{\frac{1}{2}} \eps + \ts{\frac{1}{2}} \\
&= \eps.
\end{align*}
It follows that $\al^{(k+1)}$ is continuous at $x$ in the strong
operator topology. Since this holds for all $x \in X$, $\al^{(k+1)}$
is continuous in the strong operator topology. By induction,
$\al^{(k)}$ is continuous in the strong operator topology for all
$k \geq 1$. To obtain continuity for all $k \in \Z \setminus
\set{0}$, note that $g = h^{-1}$ is also a homeomorphism, and for
any $k \geq 1$ we have
\[
\al_{x}^{(-k)} = \al_{h^{x}} \circ \cdots \circ \al_{h^{k}(x)}
= \al_{g^{-1}(x)} \circ \cdots \circ \al_{g^{-k}(x)}.
\]
Applying the above argument to the map $\gm^{(k)} \colon X \to \Aut
(A)$ given by $\gm^{(k)}(x) = \al_{x} \circ \al_{g^{-1}(x)} \circ
\al_{g^{-k}(x)}$ shows that $\gm_{x}^{(k)} = \al_{x} \circ
\al_{x}^{(-k)}$ is continuous at $x$ in the strong operator topology
for $k \geq 1$. Since $\al_{x}^{-1}$ is also continuous at $x$ in
the strong operator topology, so is $\al_{x}^{(-k)} = \al_{x}^{-1}
\circ \gm_{x}^{(k)}$ Thus $\al^{(k)}$ is continuous in the strong
operator topology for all $k \in \Z$.

Finally, $\al^{-1}$ is continuous in the strong operator topology by
Lemma \ref{AlphaInverse}, and so an argument analogous to the 
one above, with $\al^{-1}$ in place of $\al$, shows that $\al^{-(k)}$ 
is continuous in the strong operator topology for all $k \in \Z$.
Further, it is easy to see that for any $x \in X$, $\al_{x}^{(k)}
\circ \al_{x}^{-(k)} = \id_{A} = \al_{x}^{-(k)} \circ \al_{x}^{(k)}$.
\end{proof}

\begin{cor}\label{PowersofBeta}
Let $(X,h)$ and $A$ be as in Notation \ref{CtsFcnsXtoA}, and let 
$\bt \in \Aut (C(X,A))$ be the automorphism of Proposition 
\ref{AutomForCXA}. For $n \in \Z \setminus \set{0}$, the 
automorphism $\bt^{n} \in \Aut (C(X,A))$ is given explicitly by 
$\bt^{n} (f) (x) = \al_{x}^{(n)} (f \circ h^{-n} (x))$ for all $x 
\in X$.
\end{cor}

\begin{proof}
We consider first the case where $n \geq 1$, and proceed by 
induction on $n$. Observe that for all $x \in X$, we have
\[
\bt^{1} (f)(x) = \bt (f)(x) = \al_{x} ( f \circ h^{-1} (x) ) =
\al_{x}^{(1)} ( f \circ h^{-1} (x) )
\]
and so the base case holds. Next, suppose that $\bt^{n} (f)(x) =
\al_{x}^{(n)} (f \circ h^{-n} (x))$ for some $n \geq 1$. Then for
all $x \in X$, we compute
\begin{align*}
\bt^{n+1} (f)(x) &= \bt^{n} ( \bt (f) ) (x) \\
&= \al_{x}^{(n)} ( ( \bt (f) ) \circ h^{-n} (x) ) \\
&= \al_{x}^{(n)} ( \bt (f) ( h^{-n}(x)) ) \\
&= \al_{x}^{(n)} ( \al_{h^{-n}(x)} (f \circ h^{-1} (h^{-n}(x)) )) \\
&= \al_{x}^{(n)} \circ \al_{h^{-n}(x)} (f \circ h^{-1-n}(x)) \\
&= \al_{x}^{(n+1)} ( f \circ h^{-(n+1)} (x) ).
\end{align*}
It follows that the result holds for all $n \geq 1$. To extend this
result to all $n \in \Z \setminus \set{0}$, we first observe that
$\psi \in \Aut(C(X,A))$, given by $\psi (f)(x) = \al_{h(x)}^{-1} ( f
\circ h (x) )$, satisfies $\psi \circ \bt (f) (x) = f(x) = \bt \circ
\psi (f) (x)$ for all $f \in C(X,A)$ and $x \in X$, and hence $\psi 
\circ \bt = \id_{C(X,A)} = \bt \circ \psi$. This gives $\psi = 
\bt^{-1}$. Further, an induction argument entirely analogous to the 
one above shows that for $k \geq 1$, $\psi^{k} (f)(x) = 
\al_{x}^{(-k)} (f \circ h^{k} (x) )$ for all $f \in C(X,A)$ and $x 
\in X$. But $\psi = \bt^{-1}$ implies that $\bt^{-k} (f)(x) = 
\al_{x}^{(-k)} (f \circ h^{k} (x) )$ for $k \geq 1$. Letting $n = 
-k$, it follows that $\bt^{n} (f)(x) = \al_{x}^{(n)} (f \circ 
h^{-n}(x) )$ for $n < 0$.
\end{proof}

\begin{lem}\label{SubequivProjns}
Let $(X,h)$ and $A$ be as in Notation \ref{CtsFcnsXtoA}, and let 
$\al \colon X \to \Aut (A)$ be continuous in the strong operator 
topology. Assume in addition that $C(X,A)$ has order on projections 
determined by traces. Let $p_{0} \in A$ be a non-zero projection, let 
$k \in \Z$, and let $\al^{(k)}$ be as in Proposition \ref{PowersOfAlpha}. 
Then for any projection $p \in A$ with the property that 
\[
\inf_{\tau \in T(A)} \tau (p_{0}) - \sup_{\tau \in T(A)} \tau (p) > 0, 
\]
the function $q_{p,k} \colon X \to A$ given by $q_{p,k}(x) =
\al_{x}^{(k)} (p)$ is a projection in $C(X,A)$ that satisfies $q_{p,k}
\precsim 1 \ten p_{0}$.
\end{lem}

\begin{proof}
It is clear that $q_{p,k}$ is continuous, that $q_{p,k}^{*} = q_{p,k}$, 
and that $q_{p,k}^{2} = q_{p,k}$. Therefore, $q_{p,k}$ is a projection 
in $C(X,A)$. By Proposition \ref{TracesOfCXA}, in order to show that 
$\ld (q_{p,k}) < \ld (1 \ten p_{0})$ for all $\ld \in T(C(X,A))$, it suffices 
to show that $\ld (q_{p,k}) < \ld (1 \ten p_{0})$ for all $\ld$ of the 
form $\ld = \mu \ten \tau$, where $\mu \in M(X)$ and $\tau \in T(A)$. 
We first observe that for any $x \in X$, $\al_{x}^{(k)} \in \Aut (A)$ 
implies that $\tau \circ \al_{x}^{(k)} \in T(A)$. Then for any $\mu \in 
M(X)$, any $\tau \in T(A)$, and $\ld = \mu \ten \tau$, we have 
\begin{align*}
\ld (q_{p,k}) = \int_{X} \tau (q_{p,k}(x)) \; d\mu &= \int_{X} \tau 
(\al_{x}^{(k)}(p)) \; d\mu \\ 
&\leq \sup_{\sm \in T(A)} \int_{X} \sm (p) \; d\mu \\ 
&< \inf_{\sm \in T(A)} \int_{X} \sm (p_{0}) \; d\mu \\ 
&\leq \int_{X} \tau (p_{0}) \; d\mu = \ld (1 \ten p_{0}).
\end{align*}
As mentioned above, this is sufficient to imply that $\ld (q_{p,k}) 
< \ld (1 \ten p_{0})$ for all $\ld \in C(X,A)$. Since order on 
projections over $C(X,A)$ is determined by traces, we conclude 
that $q_{p} 
\precsim 1 \ten p_{0}$.
\end{proof}

\begin{lem}\label{UnitarySubequal}
Let $(X,h)$ and $A$ be as in Notation \ref{CtsFcnsXtoA}, and 
assume in addition that $C(X,A)$ has cancellation of projections. 
Let $p, q \in C(X,A)$ be projections with $p \precsim q$. Then there
is a unitary $w \in C(X,A)$ such that $w p w^{*} \leq q$.
\end{lem}

\begin{proof}
Since $C(X,A)$ has cancellation, there exists a projection $e \in 
C(X,A)$ such that $e \leq q$ and partial isometries $s,t \in C(X,A)$ 
such that $s^{*}s = p, ss^{*} = e, t^{*}t  = 1 - p$, and $tt^{*} = 1 
- e$. Define $w = s + t$. It is straightforward to check that 
$s^{*}t = st^{*} = ts^{*} = t^{*}s = 0$, from which it follows that 
$w^{*}w = (s^{*} + t^{*})(s + t) = s^{*}s + t^{*}t =p + (1-p) = 1$ 
and $ww^{*} = (s + t)(s^{*} + t^{*}) = ss^{*} + tt^{*} = e + (1 - e)
= 1$, so $w$ is unitary. Moreover,
\begin{align*}
w p w^{*} &= (s + t)p(s^{*} + t^{*}) \\
&= sps^{*} + tpt^{*} + spt^{*} + tps^{*} \\
&= ss^{*}ss^{*} + t(1 - t^{*}t)t^{*} + ss^{*}st^{*} + t(1 - t^{*}t)
s^{*} \\
&= e^{2} + tt^{*} - tt^{*}tt^{*} \\
&= e + (1 - e) - (1 - e)^{2} \\
&= e,
\end{align*}
as required. 
\end{proof}

\begin{dfn}\label{HeredSubalg}
Let $(X,h)$ and $A$ be as in Notation \ref{CtsFcnsXtoA}. 
For an open set $V \subset X$ and a projection $p \in A$, the
{\emph{hereditary subalgebra of $C(X,A)$ determined by $V$ 
and $p$}}, denoted by $\Her (V,p)$, is defined to be the 
hereditary subalgebra of $C(X,A)$ generated by all functions 
$f \in C(X,A)$ such that $\supp (f) \subset V$ and $f \leq 1 \ten p$.
\end{dfn}

We wish to show that given a hereditary subalgebra determined 
by some non-zero projection $p$ and non-empty open set $V$, 
any sufficiently small positive central element $f \ten 1$ of $C(X,A)$ 
can be decomposed into positive elements which, upon translating 
by $\bt$ and conjugating by unitaries, are mutually orthogonal 
elements of the hereditary subalgebra. To do this requires us to be 
able to carry out a similar decomposition at the level of the space 
$X$. Definition \ref{DCP}, which first appeared in \cite{Buck2}, 
gives a property which allows such a decomposition. In order to 
state it, and for some results which come later, we need definitions 
describing certain smallness properties for closed sets.

\begin{dfn}\label{UnivNullTopSmall}
Let $(X,h)$ be as in Notation \ref{CtsFcnsXtoA}, and let $F \subset X$ 
be closed.
\begin{enumerate}
\item We say $F$ is {\emph{universally null}} if $\mu (F) = 0$ for all 
$\mu \in M_{h}(X)$.
\item We say $F$ is {\emph{topologically $h$-small}}  if there is 
some $m \in \Z_{+}$ such that whenever $d(0), d(1), \ldots, d(m)$ 
are $m+1$ distinct elements of $\Z$, then $h^{d(0)}(F) \cap 
h^{d(1)}(F) \cap \cdots \cap h^{d(m)}(F) = \varnothing$.
\end{enumerate}
\end{dfn}

In Corollary 2.11 of \cite{Buck2}, it is shown that a topologically 
$h$-small set is universally null. This fact will be used repeatedly 
in the proof of Theorem \ref{TQRPforBeta}.

\begin{dfn}\label{DCP}
Let $(X,h)$ be as in Notation \ref{CtsFcnsXtoA}. 
We say $(X,h)$ has the {\emph{dynamic comparison property}} if 
whenever $U \subset X$ is open and $C \subset X$ is closed with 
$\del C$, $\del U$ universally null and $\mu (C) < \mu (U)$ for every 
$\mu \in M_{h}(X)$, then there are $M \in \N$, continuous functions 
$f_{j} \colon X \to [0,1]$ for $0 \leq j \leq M$, and $d(0),\ldots,d(M) \in 
\Z$ such that $\sum_{j=0}^{M} f_{j} = 1$ on $C$, and such that the 
sets $\supp (f_{j} \circ h^{-d(j)})$ are pairwise disjoint subsets of $U$ 
for $0 \leq j \leq M$.
\end{dfn}

In \cite{Buck2} it is shown this property holds for a large class of 
minimal dynamical systems $(X,h)$ that includes all 
finite-dimensional examples. It will also be used in the proof of 
our main result.

\begin{prp}\label{HeredSubalgDecomp}
Let $(X,h)$ and $A$ be as in Notation \ref{CtsFcnsXtoA}. 
Let $\bt \in \Aut (C(X,A))$ be the automorphism of Proposition 
\ref{AutomForCXA}. Assume that $(X,h)$ has the dynamic 
comparison property and that $A$ is a non-elementary 
$C^{*}$-algebra with real rank zero and order 
on projections determined by traces. Then for every non-zero 
projection $p_{0} \in A$ and every non-empty open set $V 
\subset X$, there exist $M \in \N$ and $\eps > 0$ such that 
whenever $g_{0} \in C(X)$ is positive and satisfies $\mu (\supp 
(g_{0}) ) < \eps$ for all $\mu \in M_{h}(X)$, then there exist for $0 
\leq k \leq M$ positive elements $a_{k} \in C(X,A)$, unitaries 
$w_{k} \in C(X,A)$, and $r(k) \in \Z$ such that: 
\begin{enumerate}
\item $\sum_{k=0}^{M} a_{k} \geq g_{0} \ten 1$; 
\item the elements $\bt^{r(k)}(a_{k})$ are mutually orthogonal, and 
$\supp (\bt^{r(k)}(a_{k})) \subset V$ for each $k$; 
\item with $b_{k} = w_{k}\bt^{r(k)} (a_{k}) w_{k}^{*}$, the $b_{k}$ 
are mutually orthogonal positive elements in $\Her (V,p_{0})$.
\end{enumerate}
\end{prp}

\begin{proof}
Set $\dt = \inf_{\tau \in T(A)} \tau (p_{0}) > 0$, and choose $N \in
\N$ such that $N > 1$ and $1/N < \dt/2$. Then by Theorem 1.1 of
\cite{Zh} there exist $2^{N}+1$ mutually orthogonal projections
$q_{0},\ldots, q_{2^{N}}$ such that $q_{0} \precsim q_{1} \sim \cdots 
\sim q_{2^{N}}$ and $\sum_{j=0}^{2^{N}} q_{j} = 1$. We immediately 
obtain $\tau (q_{1}) = \cdots = \tau (q_{2^{N}})$ for all $\tau \in 
T(A)$. Then for $1 \leq j \leq 2^{N}$ and each $\tau \in T(A)$, we 
have 
\[
1 = \tau (1) = \sum_{i=0}^{2^{N}} \tau 
(q_{i}) \geq \sum_{i=1}^{2^{N}} \tau (q_{i}) = 2^{N} \tau (q_{j}),
\]
and so $\tau (q_{j}) \leq 1/2^{N}$. This gives $\tau (q_{j}) < 1/N < 
\dt/2$ for $1 \leq j \leq 2^{N}$, and hence that 
\[
\inf_{\tau \in T(A)} \tau (p_{0}) - \sup_{\tau \in T(A)} (q_{j}) > \dt - 
\dt/2 = \dt/2 > 0
\]
for all $\tau \in T(A)$. In particular, we clearly have $\tau (q_{j}) < 
\tau (p_{0})$ for $1 \leq j \leq 2^{N}$ and for all $\tau \in T(A)$, and 
since the order on projections in $A$ is determined by traces, we 
conclude that $q_{j} \precsim p_{0}$ for $1 \leq j \leq 2^{N}$. Since 
$q_{0} \precsim q_{1}$, we actually obtain $q_{j} \precsim p_{0}$ 
for $0 \leq j \leq 2^{N}$.

Set $J = 2^{N}$, and let $\sm = \inf_{\mu \in M_{h}(X)} \mu (V) > 0$. 
Choose $J$ distinct points $x_{0},\ldots,x_{J} \in V$ and for each $j$ 
consider the nested sequence of neighborhoods 
$(B(x_{j},1/k))_{k=1}^{\infty}$. Choose $K_{J+1} \in \N$ so large that 
the sets $B(x_{j},1/K_{J+1})$ are pairwise disjoint subsets of $V$ for 
$0 \leq j \leq J$. 

For each $0 \leq j \leq J$, apply the same argument as in the proof of 
Lemma 1.4 of \cite{Buck2} to find a $K_{j} \in \N$ so large that 
$\mu ( B(x_{j},1/K_{j}) ) < \sm/(J+1)$ for every $\mu \in M_{h}(X)$.
Let $K = \max \set{K_{0},\ldots,K_{J+1}}$, and for $0 \leq j \leq 2^{N}$ 
set $V_{j} = B(x_{j},1/K)$. Then for $0 \leq j \leq J$, we have $\mu 
(V_{j}) < \sm/(J+1)$ for every $\mu \in M_{h}(X)$, and the sets $V_{j}$ 
are pairwise disjoint subsets of $V$. Using Proposition 3.9 of 
\cite{Buck2}, for $0 \leq j \leq J$ choose open sets $W_{j}$ such that 
$x_{j} \in W_{j} \subset \overline{W}_{j} \subset V_{j}$ with $\del 
W_{j}$ topologically $h$-small. Now set 
\[
\eps  = \min_{0 \leq j \leq M} \inf_{\mu \in M_{h}(X)} \mu (W_{j}) > 0.
\]
Choose an open set $E \subset X$ such that $\mu (E) < \eps$ for 
all $\mu \in M_{h}(X)$, and let $g_{0} \in C(X)$ be positive such that 
$C_{0} = \supp (g_{0}) \subset E$. Apply Proposition 3.9 of 
\cite{Buck2} to obtain a closed set $C$ with 
$C_{0} \subset C \subset E$ and $\del C$ topologically $h$-small. 
Then for $0 \leq j \leq J$, we have $\mu (C) < \eps < \mu (W_{j})$ for 
all $\mu \in M_{h}(X)$, with the sets $\del C$ and $\del W_{j}$ 
universally null. By assumption, $(X,h)$ has the dynamic comparison 
property. Thus for each $0 \leq j \leq J$ there 
exist $M_{j} \in \N$, continuous functions $f_{j,i} \colon X \to [0,1]$ for 
$0 \leq i \leq M_{j}$, and $r_{j}(i) \in \Z$ for $0 \leq i \leq M_{j}$, such 
that $\sum_{i=0}^{M_{j}} f_{j,i} = 1$ on $C$ (and hence also on 
$C_{0} = \supp (g_{0})$) and such that the sets $\supp (f_{j,i} \circ 
h^{-r_{j}(i)})$ are pairwise disjoint subsets of $W_{j} \subset V_{j}$ 
for $0 \leq i \leq M_{j}$.

For $0 \leq j \leq J$ and $0 \leq i \leq M_{j}$, define $q_{j,i}
\colon X \to A$ by $q_{j,i} (x) = \al_{x}^{(r_{j}(i))}(q_{j})$. From the 
inequality computed earlier, we have 
\[
\inf_{\tau \in T(A)} \tau (p_{0}) - \sup_{\tau \in T(A)} \tau (q_{j,i}) 
\geq \dt/2 > 0.
\]
Then by Lemma \ref{SubequivProjns}, each $q_{j,i}$ is an element
of $C(X,A)$ and $q_{j,i} \precsim 1 \ten p_{0}$. Hence by Lemma 
\ref{UnitarySubequal}, there exist unitaries $w_{j,i} \in C(X,A)$ 
for $0 \leq j \leq J$, $0 \leq i \leq M_{j}$ such that $w_{j,i} 
q_{j,i} w_{j,i}^{*} \leq 1 \ten p_{0}$.  Now for $0 \leq j \leq J$ 
and $0 \leq i \leq M_{j}$ set $a_{j,i} = f_{j,i} \ten q_{j}$ and 
$b_{j,i} = w_{j,i} \bt^{r_{j}(i)} (a_{j,i}) w_{j,i}^{*}$. 

Let $x \in X$. If $x \not \in C$, then $(g_{0} \ten 1) (x) = 0 
\leq \sum_{j=0}^{J} \sum_{i=0}^{M_{j}} a_{j,i}(x)$. If $x \in C$, 
then we compute
\[
\sum_{j=0}^{J} \sum_{i=0}^{M_{j}} a_{j,i}(x) = \sum_{j=0}^{J} 
\sum_{i=0}^{M_{j}} f_{j,i}(x) q_{j} = \sum_{j=0}^{J} q_{j} \bigg( 
\sum_{i=0}^{M_{j}} f_{j,i}(x) \bigg) = \sum_{j=0}^{J} q_{j} = 1.
\]
It follows that $g_{0} \ten 1 \leq \sum_{j=0}^{J} \sum_{i=0}^{M_{j}} 
a_{j,i}$. Next, for any $x \in X$, we have
\begin{align*}
\bt^{r_{j}(i)} (a_{j,i})(x) &= \al_{x}^{(r_{j}(i))} ( (f_{j,i}
\circ h^{-r_{j}(i)}(x) ) q_{j} ) \\
&= ( f_{j,i} \circ h^{-r_{j}(i)}(x) ) \al_{x}^{(r_{j}(i))} (q_{j}) \\
&= ( f_{j,i} \circ h^{-r_{j}(i)}(x) ) q_{j,i} (x).
\end{align*}
This gives $\supp (\bt^{r_{j}(i)} (a_{j,i})) \subset \supp (f_{j,i}
\circ h^{-r_{j}(i)} ) \subset V$ which implies the sets $\supp
(\bt^{r_{j}(i)} (a_{j,i}) )$ are pairwise disjoint, and hence that the
elements $\bt^{r_{j}(i)} (a_{j,i})$ are mutually orthogonal. Since
$\supp (b_{j,i}) \subset \supp (\bt^{r_{j}(i)} (a_{j,i}) )$, it
follows immediately that the $b_{j,i}$ are also mutually orthogonal.
Moreover, as $0 \leq f_{j,i} \leq 1$ and $w_{j,i} q_{j,i}
w_{j,i}^{*} \leq p_{0}$, it follows that $0 \leq b_{j,i} \leq 1 \ten
p_{0}$. Therefore, the $b_{j,i}$ are mutually orthogonal positive
elements in $\Her (V,p_{0})$. Now simply order the $a_{j,i}, w_{j,i},
d_{j}(i),$ and $b_{j,i}$ as $a_{k},w_{k},d(k),$ and $b_{k}$ for $0
\leq k \leq M$, where $M+1 = \sum_{j=0}^{J} M_{j}$.
\end{proof}

\begin{lem}\label{FuncCalcContProjn}
Let $E \subset \C$ be open, let $f \colon E \to \C$ be continuous,
let $A$ be a unital $C^{*}$-algebra, and set $Q = \set{b \in A
\colon b \; \textnormal{is normal with} \; \spec (b) \subset E}$.
Then $\ph \colon Q \to A$ given by $\ph (b) = f(b)$ is
norm-continuous.
\end{lem}

\begin{proof}
This is easily adapted from Lemma 2.5.11 of \cite{HLnBook}.
\end{proof}

\begin{prp}\label{HeredSubalgCorner}
Let $(X,h)$ and $A$ be as in Notation \ref{CtsFcnsXtoA}. 
Let $g \in C(X,A)$ be a non-zero positive element with $\norm{g} = 
1$. Then there is an open set $V \subset \supp(g)$, a non-zero
projection $p_{0} \in A$, and a unitary $w \in C(X,A)$ such that
$ w f w^{*} \in \overline{gC(X,A)g}$ for all $f \in \Her (V,p_{0})$.
\end{prp}

\begin{proof}
Let $\eps > 0$ be given, and assume that $\eps < 1$. Since
$\norm{g} = 1$ and $X$ is compact, there exists $x_{0} \in \supp 
(g)$ such that $\norm{ g(x_{0}) } = 1$. Let $a = g(x_{0})$ (note 
that $a \geq 0$ since $g$ is positive) and define continuous 
functions $k_{1},k_{2} \colon [0,1] \to [0,1]$ by
\[
k_{1} (t) = \begin{cases} \frac{32}{32 - \eps} t & 0 \leq t \leq 1 
- \frac{\eps}{32} \\ 1 & 1 - \frac{\eps}{32} < t \leq 1 \end{cases}
\]
and
\[
k_{2} (t) = \begin{cases} 0 & 0 \leq t \leq 1 - \frac{\eps}{64} \\ 
\frac{64}{\eps} (t - 1) + 1 & 1 - \frac{\eps}{64} < t \leq 1. 
\end{cases}
\]
Setting $a_{1} = k_{1}(a)$ and $a_{2} = k_{2}(a)$, we observe that
$a_{2} a_{1} = a_{2}$ and 
\[
\norm{ a - a_{1} } = \sup_{t \in [0,\norm{a}] } \abs{t - k_{1} (t)} 
< \ts{\frac{1}{16}} \eps.
\]
This gives $\norm{ a_{2} a - a_{2} } = \norm{ a_{2} a - a_{2} 
a_{1} } \leq \norm{ a - a_{1} } < \ts{\frac{1}{16}} \eps$. Since $A$ 
has real rank zero, there is a non-zero projection $q \in 
\overline{a_{2} A a_{2}}$. Then $a_{2} a_{1} = a_{2}$ implies that 
$q a_{1} = q$. We thus obtain $\norm{qa - q} = \norm{ qa - qa_{1} } 
\leq \norm{a - a_{1}} < \ts{\frac{1}{16}} \eps$, and similarly 
$\norm{ aq - q } < \ts{\frac{1}{16}} \eps$. Now choose a 
neighborhood $U$ of $x_{0}$ such that $\norm{ g(x) - g(x_{0}) } < 
\ts{\frac{1}{8}} \eps$ for all $x \in U$.  Using the compactness of 
$X$, choose an open set $W \subset U$ with $\overline {W} 
\subset U$, and set $K = \overline{W}$. Then for all $x \in K$,
\begin{align*}
\norm{ q g(x) - q } &\leq \norm{ q g(x) - q g(x_{0}) } + \norm{ q
g(x_{0}) - q } \\
&\leq \norm{ g(x) - g(x_{0}) } + \norm{ qa - q } \\
&< \ts{\frac{1}{8}} \eps + \ts{\frac{1}{8}} \eps \\
&= \ts{\frac{1}{4}} \eps.
\end{align*}
So for all $x \in K$, we have
\begin{align*}
\norm{ g(x) q g(x) - q } &\leq \norm{ g(x) q g(x) - g(x) q } + \norm{
g(x) q - q } \\
&\leq \norm{ g(x) } \norm{ q g(x) - q } + \norm{ g(x) q - q } \\
&< \ts{\frac{1}{4}} \eps + \ts{\frac{1}{4}} \eps \\
&= \ts{\frac{1}{2}} \eps.
\end{align*}
Set $E = (-\infty, 1/2) \cup (1/2,\infty)$, $f = 
\chi_{(1/2,\infty)}$, and $Q = \set{ b \in A \colon bb^{*} = b^{*}b, 
\spec (b) \subset E}$. Apply Lemma 
\ref{FuncCalcContProjn} to obtain a continuous function $\ph \colon 
Q \to A$ such that $\ph (b) = \chi_{(1/2,\infty)}(b)$ for all $b \in 
Q$. Next observe that for all $x \in K$, $\norm{ g(x)qg(x) - q } < 
\ts{\frac{1}{2}} \eps < \ts{\frac{1}{2}}$ implies that $g(x) q g(x) 
\in Q$. Thus we may define a function $\psi \colon K \to Q$ by 
$\psi (x) = g(x) q g(x)$. Further, for $x,y \in K$ we have
\begin{align*}
\norm{ \psi (x) - \psi (y) } &= \norm{ g(x)qg(x) - g(y)qg(y) } \\
&\leq \norm{ g(x)qg(x) - q } + \norm{ q - g(y)qg(y) } \\
&< \ts{\frac{1}{2}} \eps + \ts{\frac{1}{2}} \eps \\
&= \eps,
\end{align*}
which implies that $\psi$ is continuous on $K$. Now setting $p^{(0)}
= \ph \circ \psi$ gives a continuous function $p^{(0)} \colon K \to 
A$ with $p^{(0)} (x) = \chi_{(1/2,\infty)} (g(x)qg(x)) \in \overline{
g(x)Ag(x) }$ for all $x \in K$. Extend $p^{(0)}$ to a continuous
function $p \colon X \to A$ such that $\supp (p) \subset \supp (g)$.
Choose $\dt > 0$ so small that $\dt < 1$ and $d(x,x_{0}) < \dt$ implies 
$p(x)$ is a projection. Set $V_{0} = B(p(x_{0}),\dt)$ and $V = 
p^{-1}(V_{0})$. Then $x_{0} \in V \subset \overline{V}$, and $\norm{ 
p(x) - p(x_{0}) } \leq \ts{\frac{1}{2}} < 1$ for all $x \in 
\overline{V}$ by the continuity of $p$. Let $p_{0} = p(x_{0})$ and 
$F = \overline{V}$. 

Set $p_{F} = p \vert_{F}$ and let $e \colon F \to A$ be the constant 
function $e (x) = p_{0}$. Then $p_{F}$ and $e$ are projections in 
$C(F,A)$, and satisfy $\norm{p_{F}(x) - e(x)} = \norm{p(x) - p_{0}} 
\leq \dt$ 
for all $x \in F$. This implies that $\norm{p_{F} - e} < 1$, and so 
by Lemma 2.5.1 of \cite{HLnBook}, there is a unitary $u \in C(F,A)$ 
such that $u p_{F} u^{*} = e$ and $\norm{1 - u} \leq \sqrt{2} 
\norm{p_{F} - e}$. This norm estimate further implies that 
$\norm{1 - u} < \sqrt{2}$, and so $u \in U_{0}(C(F,A))$. (Recall that 
for a unital $C^{*}$-algebra $B$, $U_{0}(B)$ denotes the connected 
component of $U(B)$ containing $1_{B}$). Since 
the restriction map $U_{0}(C(X,A)) \to U_{0}(C(F,A))$ is surjective, 
there is a $w \in U_{0} (C(X,A))$ such that $w \vert_{F} = u$. If $f 
\in \Her (V,p_{0})$, then $\supp (f) \subset F$ and $f \leq 1 \ten 
p_{0}$. Then for any $x \in \supp (f)$, we have $w(x) f(x) w(x)^{*} 
\leq w(x) p_{0} w_{x}^{*} = u(x) p_{0} u_{x}^{*} = p(x)$. Thus for 
every $f \in \Her (V,p_{0})$, $\supp (f) \subset F \subset \supp 
(g)$ and $f(x) \in \overline{g(x)Ag(x)}$ for all $x \in X$.
\end{proof}

In order to prove our main result, we require a well-known tool 
in applications of topological dynamics to $C^{*}$-algebras: 
the Rokhlin tower construction.

\begin{thm}\label{RokhlinTower}
Let $(X,h)$ be as in Notation \ref{CtsFcnsXtoA}. 
Let $Y \subset X$ be a closed set with $\sint (Y) \neq \varnothing$.
For $y \in Y$, define $r(y) = \min \set{m \geq 1 \colon h^{m}(y) \in 
Y}$. Then $\sup_{y \in Y} r(y) < \infty$, so there are finitely many 
distinct values $n(0) < n(1) < \cdots < n(l)$ in the range of $r$. 
For $0 \leq k \leq l$, set
\[
Y_{k} = \overline{ \set{ y \in Y \colon r(y) = n(k) } } 
\hspace{0.5 in} \textnormal{and} \hspace{0.5 in} 
Y_{k}^{\circ} = \sint (\set{ y \in Y \colon r(y) = n(k) }).
\]
Then:
\begin{enumerate}
\item the sets $h^{j}(Y_{k}^{\circ})$ are pairwise disjoint for $0 \leq
k \leq l$ and $0 \leq j \leq n(k) - 1$;
\item $\bigcup_{k=0}^{l} Y_{k} = Y$;
\item $\bigcup_{k=0}^{l} \bigcup_{j=0}^{n(k)-1} h^{j}(Y_{k}) = X$.
\end{enumerate}
\end{thm}

\begin{proof}
Proofs of some or all of the statements in this theorem can be 
found in \cite{QLinPh1}, \cite{QLinPh2}, and \cite{QLinPhDiff} 
(as well as other places).
\end{proof}

We are now in position to prove that our automorphisms $\bt$ 
satisfy the tracial quasi-Rohklin property.

\begin{thm}\label{TQRPforBeta}
Let $(X,h)$ and $A$ be as in Notation \ref{CtsFcnsXtoA}. Suppose 
in addition that $(X,h)$ has the dynamic comparison property, that 
$A$ is non-elementary with real rank zero and order on projections 
determined by traces such that $C(X,A)$ has cancellation of 
projections and order on projections determined by traces, and 
that $\bt \in \Aut (C(X,A))$ is the automorphism of Proposition 
\ref{AutomForCXA}. Then $\bt$ has the tracial quasi-Rokhlin 
property.
\end{thm}

\begin{proof}
Let $\eps > 0$, let $F \subset C(X,A)$ be finite, let $n 
\in \N$, and let $g \in C(X,A)$ be positive with $\norm{g} = 1$. 
By Proposition \ref{HeredSubalgCorner}, there is non-zero projection 
$p_{0} \in A$, an open set $V \subset \supp (g)$, and a unitary $u 
\in C(X,A)$ such that $u f u^{*} \in \overline{gC(X,A)g}$ for all $f 
\in \Her (V,p_{0})$. By Proposition \ref{HeredSubalgDecomp}, there 
is an $M \in \N$ and a $\dt > 0$ such that for any positive element 
$g_{0} \in C(X)$ with $\mu (\supp (g_{0})) < \dt$ for all $\mu \in 
M_{h}(X)$, there exist for $0 \leq k \leq M$ positive elements $a_{k} \in 
C(X,A)$, unitaries $w_{k} \in C(X,A)$, and $r(k) \in \Z$ such that 
$\sum_{k=0}^{M} a_{k} \geq g_{0} \ten 1$, the elements 
$\bt^{r(k)}(a_{k})$ are mutually orthogonal, and such that with $b_{k} 
= w_{k} \bt^{r(k)} (a_{k}) w_{k}^{*}$, the $b_{k}$ are mutually orthogonal 
elements of $\Her (V,p_{0})$. By the continuity of $g$ and the 
compactness of $X$, there exist $x_{0} \in X$ with $\norm{g(x_{0})} = 1$ 
and an open neighborhood $G$ of $x_{0}$ such that $\norm{ g(x) } > 1 - 
\ts{\frac{1}{2}} \eps$ for all $x \in G$. Choose open neighborhoods 
$G_{0},G_{1},G_{2}$ of $x_{0}$ such that $G_{2} \subset 
\overline{G}_{2} \subset G_{1} \subset \overline{G}_{1} 
\subset G_{0} \subset G$, $\mu (G_{0}) < \dt$ for all $\mu \in M_{h}(X)$, 
$\del G_{2}$ is topologically $h$-small.
and $\norm{g(x)} > 1 - \eps$ for all $x \in G_{2}$. To see this can be done, 
observe that such a neighborhood $G_{0}$ exists by applying Corollary 
1.5(4) of \cite{Buck2} to the closed set 
$\set{x_{0}}$, while $G_{1}$ and $G_{2}$ exist using local compactness, 
continuity, and Proposition 3.9 of \cite{Buck2}. Choose 
continuous functions $g_{0}, g_{1} \colon X \to [0,1]$ such that 
$g_{1} = 1$ on $\overline{G}_{2}$, $\supp (g_{1}) \subset G_{1}$, 
$g_{0} = 1$ on $\overline{G}_{1}$, and $\supp (g_{0}) \subset G_{0}$. 
Apply Proposition \ref{HeredSubalgDecomp} with $g_{0}$ to obtain the 
$a_{k}$, $w_{k}$, and $r(k)$ described above. Set $\sm = \min 
\set{\ts{\frac{1}{2}} \inf_{\mu \in M_{h}(X)} \mu (G_{2}), \eps} > 0$ and 
choose $K \in \N$ so large that $\ts{\frac{1}{K}} < \ts{\frac{1}{8}} \sm$. 
Apply Lemma 4.4 of \cite{Buck2} with $N = nK$ to obtain a 
closed set $Y \subset X$ such that $\sint (Y) \neq \varnothing$, $\del Y$ 
is topologically $h$-small, and the sets $Y,h(Y),\ldots, h^{nK}(Y)$ are 
pairwise disjoint. Adopt the notation of Theorem \ref{RokhlinTower}, 
and let $M = (l + 1) \sum_{k=0}^{l} n(k)$. Then:
\begin{enumerate}
\item the sets $h^{j}(Y_{k}^{\circ})$ are pairwise disjoint for $0 
\leq k \leq l$ and $0 \leq j \leq n(k) - 1$;
\item $\bigcup_{k=0}^{l} Y_{k} = Y$;
\item $\bigcup_{k=0}^{l} \bigcup_{j=0}^{n(k)-1} h^{j}(Y_{k}) = X$;
\item $\del h^{j}(Y_{k})$ is topologically $h$-small for $0 \leq k 
\leq l$ and $0 \leq j \leq n(k) - 1$;
\item for $0 \le k \leq l$, there exists an open set $U_{k} \subset
Y_{k}^{\circ}$ such that $\overline{U}_{k} \subset Y_{k}^{\circ}$, 
$\del U_{k}$ is topologically $h$-small, and $\mu (Y_{k}^{\circ} 
\setminus \overline{U}_{k}) < \ts{\frac{\sm}{8M}}$ for all $\mu \in 
M_{h}(X)$;
\item for $0 \leq k \leq l$, there exists an open set $W_{k} \subset
U_{k}$ such that $\overline{W}_{k} \subset U_{k}$, $\del W_{k}$ is 
topologically $h$-small, and $\mu (U_{k} \setminus 
\overline{W}_{k}) < \ts{\frac{\sm}{8M}}$ for all $\mu \in M_{h}(X)$.
\end{enumerate}
Properties (1)-(3) follow immediately from Theorem \ref{RokhlinTower}, 
and property (4) is given by Lemma 4.5 of \cite{Buck2}. For (5), we 
observe that $\del Y_{k}^{o} = \del Y_{k}$ has $\mu (\del Y_{k}^{o}) = 0$ 
for all $\mu \in M_{h}(X)$, and then, for each $k$, apply Corollary 
1.5(3) of \cite{Buck2} to $Y_{k}^{o}$, 
obtaining open sets $U_{k}^{(0)}$ such that $U_{k}^{(0)} \subset 
\overline{U}_{k}^{(0)} \subset Y_{k}^{o}$ and $\mu (Y_{k}^{o} 
\setminus \overline{U}_{k}^{(0)}) < \frac{\sm}{8M}$ for all $\mu \in 
M_{h}(X)$. Now, apply Proposition 3.9 of \cite{Buck2} to obtain, for 
each $k$, an open set $U_{k}$ with $\overline{U}_{k}^{(0)} 
\subset U_{k} \subset \overline{U}_{k} \subset Y_{k}^{o}$ and such 
that $\del U_{k}$ is topologically $h$-small. (Notice that $U_{k}$ plays 
the role of the open set $V$ there, and we ignore the other conclusions.) 
Then $\overline{U}_{k}^{(0)} \subset U_{k}$ clearly implies that 
$\mu (Y_{k}^{o} \setminus \overline{U}_{k}) \leq \mu (Y_{k}^{o} 
\setminus \overline{U}_{k}^{(0)}) < \tfrac{\sm}{8M}$ for all $\mu \in 
M_{h}(X)$. The argument to obtain (6) is entirely analogous to that for 
(5), with $U_{k}, W_{k}^{(0)}$, and $W_{k}$ in place of $Y_{k}^{o}, 
U_{k}^{(0)}$, and $U_{k}$ respectively.

Now for $0 \leq k \leq l$ set $s(k) = \max \set{ m \geq 1 \colon mn 
\leq n(k) - 1 }$. Note that $s(k) \geq K$ by the choice of $Y$. For $0 
\leq k \leq l$ and $0 \leq j \leq s(k)$, choose continuous functions 
$c_{k,j}^{(0)} \colon X \to [0,1]$ such that $c_{k,j}^{(0)} = 1$ on 
$h^{jn} (\overline{W}_{k})$, and $\supp (c_{k,j}^{(0)}) \subset 
h^{jn}(U_{k})$. Next set $c_{k,j} = c_{k,j}^{(0)} \ten 1$ for $0 \leq k
\leq l$ and $0 \leq j \leq s(k)$. Finally, define $c_{0},\ldots,
c_{n} \in C(X,A)$ by setting
\[
c_{0} = \sum_{k=0}^{l} \sum_{i=0}^{s(k)} c_{k,i}
\]
and $c_{j+1} = \bt (c_{j})$ for $0 \leq j \leq n-1$. It follows
immediately from these definitions that:
\begin{enumerate}
\item $0 \leq c_{j} \leq 1$ for $0 \leq j \leq n$;
\item $c_{j}c_{k} = 0$ for $0 \leq j,k \leq n$ and $j \neq k$;
\item $\norm{ \bt (c_{j}) - c_{j+1} } = 0$ for $0 \leq j \leq n-1$;
\item $\norm{ c_{j} f - f c_{j} } = 0$ for $0 \leq j \leq n$ and for
all $f \in F$.
\end{enumerate}
Now set $c = \sum_{j=0}^{n} c_{j}$ and $C = \supp (1 - c)$. Also 
set 
\[
F = X \setminus \bigcup_{k=0}^{l} \bigcup_{j=0}^{s(k)n}
h^{j}(W_{k}).
\]
Then we immediately have $C \subset F$, and we observe 
that 
\[
\del F = \del \bigcup_{k=0}^{l} \bigcup_{j=0}^{s(k)n}
h^{j}(W_{k}) \subset \bigcup_{k=0}^{l} \bigcup_{j=0}^{s(k)n}
\del h^{j}(W_{k}).
\]
Since each set $h^{j}(W_{k})$ is topologically $h$-small, it 
follows that $F$ is topologically $h$-small by Lemma 2.3(3) 
of \cite{Buck2}. Also, $\del Y_{k}$ topologically $h$-small for 
$0 \leq k \leq l$ implies that $\mu (\del Y_{k}) = 0$ for all $\mu 
\in M_{h}(X)$, and so $\mu (Y_{k}) = \mu (Y_{k}^{\circ})$. 
Since the $Y_{k}^{\circ}$ are pairwise disjoint, we obtain for 
each $\mu \in M_{h}(X)$ the inequality
\[
\mu (Y) = \mu \bigg( \bigcup_{k=0}^{l} Y_{k} \bigg) \geq \mu 
\bigg( \bigcup_{k=0}^{l} Y_{k}^{\circ} \bigg) = \sum_{k=0}^{l} \mu 
(Y_{k}^{\circ}) = \sum_{k=0}^{l} \mu (Y_{k}).
\]
Further, the $h$-invariance of $\mu$ and the pairwise disjointness
of the sets $h^{j}(Y)$ for $0 \leq j \leq nK$ imply that
\[
1 \geq \sum_{j=0}^{nK} \mu (h^{j}(Y)) = \sum_{j=0}^{nK} \mu (Y)
= nK \mu (Y)
\]
for every $\mu \in M_{h}(X)$, and so we have $\mu (Y) < 1/(nK)$. 
Observing that $\mu (\del U_{k}) = mu (\del W_{k}) = 0$ for all $\mu 
\in M_{h}(X)$, it follows that, for any $\mu \in M_{h}(X)$,
\begin{align*}
\mu (F) &\leq \mu \bigg( X \setminus \bigcup_{k=0}^{l}
\bigcup_{j=0}^{s(k)n} h^{j}(W_{k}) \bigg) \\
&\leq \sum_{k=0}^{l} \sum_{j=s(k)n+1}^{n(k)-1} \mu ( h^{j}(Y_{k}) )
+ \sum_{k=0}^{l} \sum_{j=0}^{s(k)n} \left( \mu( h^{j} ( U_{k}
\setminus W_{k} ) ) + \mu ( h^{j} ( Y_{k} \setminus U_{k} ) )
\right) \\
&= \sum_{k=0}^{l} \sum_{j=s(k)n+1}^{n(k)-1} \mu ( Y_{k} )
+ \sum_{k=0}^{l} \sum_{j=0}^{s(k)n} \left( \mu( U_{k}
\setminus W_{k} ) + \mu ( Y_{k} \setminus U_{k} ) \right) \\
&\leq (n+1) \mu(Y) +  M \Big( \frac{\sm}{8M} + \frac{\sm}{8M} \Big) \\
&< \frac{n+1}{nK} + \ts{\frac{1}{4}} \sm \\
&< \frac{2}{K} + \ts{\frac{1}{4}} \sm \\
&< \ts{\frac{1}{2}} \sm.
\end{align*}
Thus for all $\mu \in M_{h}(X)$ we have 
\[
\mu (F) < \sm < \inf_{\mu \in M_{h}(X)} \mu (G) \leq \mu (G_{2}),
\] 
with both $\del F$ and $\del G_{2}$ topologically $h$-small, and 
hence universally null. Then by the dynamic comparison property 
there exist $N \in \N$, continuous functions $f_{j}^{(0)} \colon X \to 
[0,1]$ for $0 \leq j \leq N$, and $d(0),\ldots,d(N) \in \Z$ such that 
$\sum_{j=0}^{N} f_{j}^{(0)} = 1$ on $F$, and such that the sets 
$\supp (f_{j}^{(0)} \circ h^{-d(j)})$ are pairwise disjoint subsets of 
$G_{2}$ for $0 \leq j \leq N$. Define continuous functions $f_{j} 
\colon X \to A$ by $f_{j} = f_{j}^{(0)} \ten 1$. Then $1 - c \leq 
\sum_{j=0}^{N} f_{j}$ since $C \subset F$, and for $0 \leq j \leq N$,
the elements $\bt^{d(j)} (f_{j})$ are mutually orthogonal positive
elements in $\overline{(g_{1} \ten 1)C(X,A)(g_{1} \ten 1)}$. For $0 
\leq j \leq N$ and $0 \leq k \leq M$, define $e_{j,k} = f_{j} 
\bt^{-d(j)} (a_{k})$. Since the $\bt^{d(j)} (f_{j})$ are mutually 
orthogonal elements of $\overline{(g_{1} \ten 1) C(X,A) (g_{1} \ten 
1)}$, it follows that $\sum_{j=0}^{N} \bt^{d(j)} (f_{j} \ten 1) \leq 
g_{0} \ten 1$. Moreover, since $\bt^{d(j) + r(k)} (e_{j,k}) = 
\bt^{r(k) + d(j)}
(f_{j}) \bt^{r(k)} (a_{k})$ and the $f_{j}$ are central, the elements
$\bt^{d(j) + r(k)} (e_{j,k})$ are mutually orthogonal. Now let
$u_{j,k} = uw_{k}$ for $0 \leq j \leq N$, $0 \leq k \leq M$. Then
\[
u_{j,k} \bt^{d(j) + r(k)} (e_{j,k}) u_{j,k}^{*} = \bt^{d(j) +
r(k)} (f_{j}) u w_{k} \bt^{r(k)} (a_{k}) w_{k}^{*} u^{*} = \bt^{d(j)
+ r(k)} (f_{j}) u b_{k} u^{*}.
\]
Since $\bt^{d(j) + r(k)} (f_{j}) \in C(X)$ and $u b_{k} u^{*} \in
\overline{gC(X,A)g}$, it follows that $u_{j,k} e_{j,k} u_{j,k}^{*} 
\in \overline{gC(X,A)g}$. Finally, we compute
\begin{align*}
\sum_{j=0}^{N} \sum_{k=0}^{M} e_{j,k} = \sum_{j=0}^{N} 
\sum_{k=0}^{M} f_{j} \bt^{-d(j)} (a_{k})
&= \sum_{j=0}^{N} f_{j} \bt^{-d(j)} \bigg( \sum_{k=0}^{M} a_{k}
\bigg) \\
&\geq \sum_{j=0}^{N} f_{j} \bt^{-d(j)} (g_{0} \ten 1) \\ 
&= \sum_{j=0}^{N} \bt^{-d(j)} ( \bt^{d(j)}(f_{j}) (g_{0} \ten 1) ) \\
&= \sum_{j=0}^{N} \bt^{-d(j)} ( \bt^{d(j)} (f_{j}) ) \\ 
&= \sum_{j=0}^{N} f_{j} \geq 1 - c.
\end{align*}
Now re-order the elements $e_{j,k}, 
u_{j,k}$, and $d(j) + r(k)$ as $e_{i}, u_{i}$, and $t(i)$ for $0 
\leq i \leq I$, where $I = (M+1)(N+1)$. It follows that $1 - c \leq 
\sum_{i=0}^{I} e_{i}$, $\bt^{t(i)}(e_{i}) \bt^{t(j)}(e_{j}) = 0$ for 
$0 \leq i,j \leq I$ and $i \neq j$, and $u_{i} e_{i} u_{i}^{*} \in 
\overline{ gC(X,A)g }$ for $0 \leq i \leq I$. Finally, as 
$\inf_{\mu \in M_{h}(X)} [\mu(G_{2}) - \mu (C)] \geq \inf_{\mu 
\in M_{h}(X)} [\mu (G_{2}) - \mu (F) ] > 0$, there is an 
$x \in G_{2}$ such that $x \not \in C$. Then $(1- c)(x) = 0$, and 
so $c(x) = 1$. It follows that $\norm{ c(x) g(x) c(x) } = \norm{g(x)} 
> 1 - \eps$, which implies that $\norm{c g c} > 1 - \eps$. Thus, 
$\bt$ has the tracial quasi-Rokhlin property.
\end{proof}

In order to apply our structure theorems from Section 
\ref{SectionTQRP} to $C^{*}(\Z,C(X,A),\bt)$, we require information
about the ideals of $C(X,A)$.

\begin{lem}\label{IdealsofCXA}
Let $(X,h)$ and $A$ be as in Notation \ref{CtsFcnsXtoA}. 
If $F \subset X$ is closed, then $I_{F} = \set{ f \in C(X,A) \colon 
f \vert_{F} = 0 }$ is an ideal in $C(X,A)$. Moreover, given any 
ideal $I \subset C(X,A)$, $I = I_{F}$ for some closed set $F \subset 
X$.
\end{lem}

\begin{proof}
For $F \subset X$ closed, it is clear that $I_{F}$ is an ideal in 
$C(X,A)$. Now let $I \subset C(X,A)$ be an 
ideal. Define $F \subset X$ by $F = \set{x \in X \colon f(x) = 0 \; 
\textnormal{for all} \; f \in I }$, which is certainly a closed 
subset of $X$. Set $I_{F} = \set{ f \in C(X,A) \colon f 
\vert_{F} = 0 }$, which we have already shown is an ideal of 
$C(X,A)$. From the definition of $F$ it is clear that $I \subset 
I_{F}$. To prove the converse, let $x_{0} \in X \setminus F$. We 
claim that $\set{ g(x_{0}) \colon g \in I }$ is dense in $A$. To see 
this, let $\dt > 0$ be given, and let $a \in A$. Since $x_{0} \not 
\in F$, there is a function $g_{0} \in I$ such that $g_{0}(x_{0}) 
\neq 0$. Then the ideal $\overline{A g_{0}(x_{0}) A}$ is non-zero and
so equals $A$ by the simplicity of $A$. It follows that there exist
$b_{1},\ldots,b_{n},c_{1},\ldots,c_{n} \in A$ such that $\norm{ a -
\sum_{j=1}^{n} b_{j} g_{0}(x_{0}) c_{j} } < \dt$. Define a function
$g \in C(X,A)$ by $g = \sum_{j=1}^{n} (1 \ten b_{j}) g_{0} (1 \ten
c_{j})$. Then $f \in I$ as $g_{0} \in I$ and $1 \ten b_{j}, 1 \ten
c_{j} \in C(X,A)$, and $\norm{g_{x_{0}} - a} < \dt$. Now let $\eps 
> 0$ be given and let $q \in I_{F}$. For each $x \in X$, choose 
$f_{x} \in I$ such that $\norm{ f_{x}(x) - q(x) } < \ts{\frac{1}{4}} 
\eps$. This can be done by taking $f_{x} = 0$ whenever $x \in F$, 
and for $x \not \in F$, $f_{x}$ can be obtained from the previous 
claim. Next for each $x \in X$ choose an open neighborhood $U_{x}$ 
of $x$ such that $\norm{ f_{x}(x) - f_{x}(y) } < \ts{\frac{1}{4}} 
\eps$ and $\norm{ q(x) - q(y) } < \ts{\frac{1}{4}} \eps$ for all $y 
\in U_{x}$. We obtain an open cover $\set{ U_{x} \colon x \in X }$ 
of $X$, which has a finite subcover $\set{U_{x_{1}},\ldots,
U_{x_{N}}}$. Let $f_{1},\ldots,f_{n}$ be the functions corresponding 
to the points $x_{1},\ldots,x_{n}$. Choose a partition of unity 
$\ph_{1},\ldots,\ph_{N}$ subordinate to this cover, let $g_{j} = 
\ph_{j} f_{j}$ for $1 \leq j \leq N$, and set $g = \sum_{j=1}^{N} 
g_{j}$. Then $g \in I$, and for $1 \leq j \leq N$ and every $x \in 
X$ we have
\begin{align*}
\norm{q(x) - f_{j}(x)} &\leq \norm{q(x) - q(x_{j})} + 
\norm{q(x_{j}) - f_{j}(x_{j})} + \norm{f_{j}(x_{j}) - f_{j}(x)} \\ 
&< \ts{\frac{1}{4}} \eps + \ts{\frac{1}{4}} \eps + 
\ts{\frac{1}{4}} \eps \\ 
&= \ts{\frac{3}{4}} \eps.
\end{align*}
For $x \in X$, let $J(x) = \set{j \colon x \in U_{j}}$. Then for 
every $x \in X$, we have
\begin{align*}
\norm{ q(x) - g(x) } &= \bigg\lVert q(x) - \sum_{j=1}^{N} \ph_{j}(x)
f_{j}(x) \bigg\rVert \\
&= \bigg\lVert \sum_{J(x)} \ph_{j}(x) q(x) - \sum_{J(x)} \ph_{j}(x)
f_{j}(x) \bigg\rVert \\
&= \bigg\lVert \sum_{J(x)} \ph_{j}(x) ( q(x) -f_{j}(x) ) \bigg\rVert \\ 
&\leq \sum_{J(x)} \ph_{j}(x) \norm{q(x) - f_{j}(x)} \\ 
&\leq \bigg( \sum_{J(x)} \ph_{j}(x) \bigg) \max_{J(x)} \set{ 
\norm{q(x) - f_{j}(x)} } \\ 
&< \ts{\frac{3}{4}} \eps.
\end{align*}
It follows that $\norm{q - f} < \eps$, and hence $q \in I$ as $I$ is 
closed. Therefore $I_{F} \subset I$, which completes the proof.
\end{proof}

\begin{prp}\label{NoBetaInvIdeals}
Let $(X,h)$ and $A$ be as in Notation \ref{CtsFcnsXtoA}. Then the 
$C^{*}$-algebra $C(X,A)$ has no non-trivial $\bt$-invariant 
ideals.
\end{prp}

\begin{proof}
Let $I \subset C(X,A)$ be a non-trivial ideal. By Lemma 
\ref{IdealsofCXA}, there is a closed set $F \subset X$ such that $I 
= \set{f \in C(X,A) \colon f(x) = 0 \; \textnormal{for all} \; x \in 
F }$. Then $F \neq \varnothing$ and $F \neq X$ as $I$ is non-trivial. 
Suppose that $I$ is $\bt$-invariant. Then $\bt (I) \subset I$, and 
so for any $f \in I$, we have $\bt (f) \in I$. Then for any $x \in 
F$, $f(x) = 0$ and $\bt (f) (x) = 0$. But $0 = \bt (f) (x) = \al_{x} 
(f \circ h^{-1}(x))$ implies that $f \circ h^{-1}(x) = 0$ since 
$\al_{x} \in \Aut (A)$. Thus $f(x) = 0$ for all $x \in F \cap 
h^{-1}(F)$. The $\bt$-invariance of $I$ further implies that 
$\bt^{n}(f) \in I$ for all $n \in N$, and recalling that $\bt^{n}(f) 
(x) = \al_{x}^{(n)} (f \circ h^{-n}(x) )$ (this is Corollary 
\ref{PowersofBeta}) and that $\al^{(n)} \in \Aut (A)$, it follows 
that for any $f \in I$, we have $f(x) = 0$ for all $x \in 
\bigcup_{n=0}^{\infty} h^{-n}(F)$. By assumption $F$ is closed and 
non-empty, and so the minimality of $h$ gives 
$\bigcup_{n=0}^{\infty} h^{-n}(F) = X$. Thus $f(x) = 0$ for all $x 
\in X$, which implies $f = 0$. It follows that $I = 0$, a 
contradiction. Therefore $I$ cannot be $\bt$-invariant, and the 
desired result follows.
\end{proof}

\begin{cor}\label{BetaCrossProdSimple}
Let $(X,h)$, $A$, and $\bt$ be as in Theorem \ref{TQRPforBeta}. 
Then the crossed product $C^{*}$-algebra $C^{*}(\Z,C(X,A),\bt)$ is 
simple.
\end{cor}

\begin{proof}
By Proposition \ref{NoBetaInvIdeals}, $C(X,A)$ has no non-trivial
$\bt$-invariant ideals. Since $\bt$ has the tracial quasi-Rokhlin
property, Theorem \ref{SimpleCrossProd} implies that
$C^{*}(\Z,C(X,A),\bt)$ is simple.
\end{proof}

\begin{dfn}\label{TopologicalScatter}
A topological space $X$ is {\emph{topologically scattered}} if every
closed subset $Y$ of $X$ contains a point $y$ that is relatively
isolated in $Y$.
\end{dfn}

It is a standard result (see \cite{PS}) that a compact Hausdorff
space $X$ is topologically scattered if and only if every Radon
measure on $X$ is atomic; that is, if and only if for any Radon 
measure $\nu$ on $X$, there exist point-mass measures 
$(\nu_{j})_{j=1}^{\infty}$ and real numbers 
$(t_{j})_{j=1}^{\infty}$, satisfying $t_{j} \geq 0$ for all $j \geq 
1$ and $\sum_{j=1}^{\infty} t_{j} = 1$, such that 
\[
\nu = \sum_{j=1}^{\infty} t_{j} \nu_{j}.
\]
Definition \ref{ScatteredAlgebra} can be thought of as a 
noncommutative version of this one, with an atomic state playing 
the role of a ``noncommutative atomic Radon measure''.

\begin{prp}\label{NonTopScatSpaces}
Given any infinite compact metrizable space $X$ that has a minimal
homeomorphism $h \colon X \to X$ and any simple, separable, unital
$C^{*}$-algebra $A$, the $C^{*}$-algebra $C(X,A)$ is not scattered.
\end{prp}

\begin{proof}
First note that as $X$ has a minimal homeomorphism, it cannot be
topologically scattered. Indeed if we take $Y = X$, then for $X$ to
be topologically scattered it must contain at least one isolated
point $y$, which is impossible since the $h$-orbit of $y$ is dense 
in $X$. Therefore $X$ has a non-atomic radon measure
$\nu$. Define a state $\ph_{\nu}$ on $C(X)$ by
\[
\psi_{\nu} (f) = \int_{X} f \; d \nu.
\]
We claim that $\psi_{\nu}$ is a non-atomic state. If it were atomic, 
we could write $\psi_{\nu} = \sum_{i=1}^{\infty} \dt_{i} \ph_{i}$ for 
some sequence of pure states $(\ph_{i})_{i=1}^{\infty}$ and some 
sequence of nonnegative real numbers $(\dt_{i})_{i=1}^{\infty}$ such 
that $\sum_{i=1}^{\infty} \dt_{i} = 1$. By the Riesz Representation 
Theorem, we would obtain $\nu = \sum_{i=1}^{\infty} \nu_{i}$ for 
some sequence of point-mass measures $\nu_{i}$, a contradiction. Now 
let $\om$ be any non-zero state on $A$, and suppose the state 
$\psi_{\nu} \ten \om$ is atomic. By Theorem IV.4.14 of \cite{Tk1}, 
we may write $\psi_{\nu} \ten \om = \sum_{i=1}^{\infty} t_{i} ( 
\ph_{i} \ten \om_{i} )$ for some sequences of pure states 
$(\ph_{i})_{i=1}^{\infty}$ on $C(X)$ and $(\om_{i})_{i=1}^{\infty}$ 
on $A$, and for some sequence of nonnegative real numbers 
$(t_{i})_{i=1}^{\infty}$ such that $\sum_{i=1}^{\infty} 
t_{i} = 1$. Then for any $f \in C(X)$, we have
\[
( \psi_{\nu} \ten \om ) (f \ten 1) = \sum_{i=1}^{\infty} t_{i}
\ph_{i} (f)
\]
which implies that $\psi_{\nu} = \sum_{i=1}^{\infty} t_{i} \ph_{i}$,
a contradiction to $\psi_{\nu}$ being non-atomic.
\end{proof}

\begin{cor}\label{BetaCrossProdTraces}
Let $(X,h)$, $A$, and $\bt$ be as in Theorem \ref{TQRPforBeta}. 
Then the restriction map $T ( C^{*}(\Z,C(X,A),\bt) ) \to T_{\bt} 
(C(X,A))$ is a bijection.
\end{cor}

\begin{proof}
By Proposition \ref{NonTopScatSpaces}, $C(X,A)$ is not a scattered
$C^{*}$-algebra, and by Proposition \ref{NoBetaInvIdeals}, $C(X,A)$ 
has no $\bt$-invariant ideals. Since $\bt$ has the tracial 
quasi-Rokhlin property, the given bijection follows from Theorem 
\ref{TracialStateSpace}.
\end{proof}

In the case where the homeomorphism $h$ is uniquely ergodic 
(there is a unique $h$-invariant Borel probability measure on $X$, 
and hence a unique $h$-invariant tracial state on $C(X)$) and 
where $A$ has a unique tracial state, we obtain the following nice 
corollary.

\begin{cor}\label{UniqueTrace}
Let $(X,h)$, $A$, and $\bt$ be as in Theorem \ref{TQRPforBeta}, and 
assume that $h$ is uniquely ergodic and that $A$ has a unique 
tracial state. Then $C^{*}(\Z,C(X,A),\bt)$ has a unique tracial state.
\end{cor}

We summarize the results of this section for crossed product 
$C^{*}$-algebras by automorphisms with the tracial quasi-Rokhlin 
property.

\begin{thm}\label{CrossProdCXA}
Let $X$ be an infinite compact metrizable space, let $h \colon X \to 
X$ be a minimal homeomorphism, and let $A$ be a simple, 
separable, unital, non-elementary $C^{*}$-algebra with real rank 
zero, such that $C(X,A)$ has cancellation of projections and order 
on projections determined by traces. Let $\bt \in \Aut (C(X,A))$ be 
defined as in Proposition \ref{AutomForCXA}. Suppose that $(X,h)$ 
has the dynamic comparison property. Then the crossed product 
$C^{*}$-algebra $C^{*}(\Z,C(X,A),\bt)$ is simple, and there is a 
bijection $T(C^{*}(\Z,C(X,A),\bt)) \to T_{\bt}(C(X,A))$. 
\end{thm}

If $A = \C$, then $C^{*}(\Z,C(X,A),\bt)$ is just $C^{*}(\Z,X,h)$, 
whose structure has been extensively studied in \cite{QLinPhDiff}, 
\cite{HLinPh}, and \cite{TomsWinter} (among other places), as 
discussed in the Introduction. In particular, in the case where $X$ 
has finite covering dimension and projections separate traces, 
then $C^{*}(\Z,X,h)$ has tracial rank zero. Even if one omits the 
assumption that projections separate traces, one of the main 
theorems of \cite{TomsWinter} shows that $C^{*}(\Z,X,h)$ is stable 
under tensoring with the Jiang-Su algebra. Whether these results 
can be extended to the situation we have studied in this paper is 
a question of considerable interest. Hua (see \cite{Hua}) has 
shown that in the case where $X$ is the Cantor set and $A$ has 
tracial rank zero, the crossed product $C^{*}(\Z,C(X,A),\bt)$ has 
tracial rank zero under some additional technical assumptions. 
However, the case for more general spaces $X$ remains open. 
Aspects related to this will be considered in subsequent papers 
\cite{Buck1} and \cite{BT}.


\begin{thebibliography}{99}

\bibitem{AkeSc} C.\ A.\ Akemann and F.\ W.\ Schultz,
{\emph{Perfect $C^{*}$-algebras}},
Memoirs of the American Mathematical Society, Number 326, vol.~55,
May 1985.


\bibitem{ArTrpPrfree} D.\ Archey,
{\emph{Crossed product $C^{*}$-algebras by finite group actions with
the projection-free tracial Rokhlin property}},
(arXiv:math.OA/O0902.3324v1).


\bibitem{Bl} B.\ Blackadar,
{\emph{Comparison theory for simple $C^{*}$-algebras}}, pages 21--54 
in: {\emph{Operator Algebras and Applications}},
D.~E.\ Evans and M.~Takesaki (eds.)
(London Math.\ Soc.\ Lecture Notes Series no.~135),
Cambridge University Press, Cambridge, New York, 1988.


\bibitem{Buck1} J.\ Buck, 
{\emph{Large subalgebras of certain crossed product 
$C^{*}$-algebras}}, 
in preparation.


\bibitem{Buck2} J.\ Buck,
{\emph{Smallness and comparison properties for minimal 
dynamical systems}}, 
in preparation.


\bibitem{BT} J.\ Buck and A.\ Tikuisis, 
{\emph{A criterion for $\JS$-stability with applications to crossed 
products}}, 
in preparation.


\bibitem{ELPW} S.\ Echterhoff, W.\ L{\"{u}}ck, N.\ C.\ Phillips,
and S.\ Walters,
{\emph{The structure of crossed products of irrational rotation
algebras by finite subgroups of $\mathrm{SL}_{2}(\Z)$}},
J.\ reine angew.\ Math., {\textbf{639}}(2010), 173--221.


\bibitem{GPS} T.\ Giordano, I.\ F.\ Putnam, and C.\ F.\ Skau,
{\emph{Topological orbit equivalence and $C^{*}$-crossed products}},
J.\ reine angew.\ Math.\ {\textbf{469}}(1995), 51--111.


\bibitem{Hua} J,\ Hua,
{\emph{Crossed products by $\alpha$-simple automorphisms on 
$C^{*}$-algebras $C(X, A)$}},
preprint (arXiv:0910.3299v2 [math.OA]).


\bibitem{Iz} M.\ Izumi,
{\emph{The Rohlin property for automorphisms of $C^{*}$-algebras}},
pages 191--206 in:
{\emph{Mathematical Physics in Mathematics and Physics (Siena, 
2000)}}, Fields Inst.\ Commun.\ vol.~30, Amer.\ Math.\ Soc., 
Providence RI, 2001.


\bibitem{Ks1} A.\ Kishimoto,
{\emph{The Rohlin property for automorphisms of UHF algebras}},
J.\ reine angew.\ Math.\ {\textbf{465}}(1995), 183--196.


\bibitem{Ks2} A.\ Kishimoto, {\emph{The Rohlin property for shifts
on UHF algebras and automorphisms of Cuntz algebras}},
J.\ Funct.\ Anal.\ {\textbf{140}}(1996), 100--123.


\bibitem{Ks3} A.\ Kishimoto, {\emph{Automorphisms of
A${\mathbb{T}}$ algebras with the Rohlin property}},
J.\ Operator Theory {\textbf{40}}(1998), 277--294.


\bibitem{La} A.\ Lazar,
{\emph{On scattered $C^{*}$-algebras}},
preprint.


\bibitem{HLnBook} H.\ Lin,
{\emph{An Introduction to the Classification of Amenable
$C^{*}$-Algebras}},
World Scientific, Singapore, 2001.


\bibitem{HLnTAF} H.\ Lin, {\emph{Tracially AF $C^{*}$-algebras}},
Trans.\ Amer.\ Math.\ Soc.\ {\textbf{353}}(2001), 693--722.


\bibitem{HLnTTR} H.\ Lin,
{\emph{The tracial topological rank of C*-algebras}},
Proc.\ London\ Math.\ Soc.\ {\textbf{83}}(2001), 199--234.


\bibitem{HLnTRZ} H.\ Lin.,
{\emph{Classification of simple $C^{*}$-algebras with tracial
topological rank zero}},
Duke Math.\ J.\ {\textbf{125}} No.\ 1 (2004), 91--114.


\bibitem{LO} H.\ Lin and H.\ Osaka,
{\emph{The Rokhlin property and the tracial topological rank}},
J.\ Funct.\ Anal.\ {\textbf{218}}(2005), 475--494.


\bibitem{HLinPh} H.\ Lin and N.\ C.\ Phillips,
{\emph{Crossed products by minimal homeomorphisms}},
J.\ reine angew.\ Math.\ {\textbf{641}}(2010), 95--122.


\bibitem{QLin} Q.\ Lin,
{\emph{Analytic structure of the transformation group $C*$-algebra
associated with minimal dynamical systems}},
preprint.


\bibitem{QLinPh1} Q.\ Lin and N.\ C.\ Phillips,
{\emph{Ordered K-theory for $C^{*}$-algebras of minimal
homeomorphisms}},
pages 289--314 in: {\emph{Operator Algebras and Operator Theory}},
L.\ Ge, et al (eds.), Contemporary Mathematics vol.~228, 1998.
MR1667666 (2000a:46118).


\bibitem{QLinPh2} Q.\ Lin and N.\ C.\ Phillips,
{\emph{Direct Limit Decomposition for $C^{*}$-algebras of minimal
diffeomorphisms}},
pages 107--133 in: {\emph{Operator Algebras and Applications}},
Advanced Studies in Pure Mathematics vol.~38, Mathematical Society
of Japan, 2004. MR2059804 (2005d:46144).


\bibitem{QLinPhDiff} Q.\ Lin and N.\ C.\ Phillips,
{\emph{The structure of $C^{*}$-algebras of minimal 
diffeomorphisms}}, 
in preparation.


\bibitem{MP} M.\ Martin and C.\ Pasnicu,
{\emph{Some comparability results in inductive limit
$C^{*}$-algebras}},
J.\ Operator Theory {\textbf{30}}(1993), 137--147.


\bibitem{Mu} G.\ J.\ Murphy,
{\emph{$C^{*}$-Algebras and Operator Theory}},
Academic Press, Boston, San Diego, New York, London, Sydney, Tokyo,
Toronto, 1990.


\bibitem{OPTrp} H.\ Osaka and N.\ C.\ Phillips,
{\emph{Stable and real rank for crossed products by automorphisms
with the tracial Rokhlin property}},
Ergod.\ Th.\ Dynam.\ Sys.\ (2006), {\textbf{26}}, 1579--1621. 


\bibitem{PS} A.\ Pelcznski and Z.\ Semadeni,
{\emph{Spaces of continuous functions III. Spaces $C(L)$ for $L$ 
without perfect subsets}}, 
Studia Math.\ (1959) {\textbf{18}}, 211--222.


\bibitem{PhRsha2} N.\ C.\ Phillips,
{\emph{Cancellation and stable rank for direct limits of recursive
subhomogeneous algebras}},
Trans.\ Amer.\ Math.\ Soc.\ {\textbf{359}} No.\ 10 (2007), 
4625--4652.


\bibitem{PhTrpfg} N.\ C.\ Phillips,
{\emph{The tracial Rokhlin property for actions of finite groups on
C*-algebras}},
in preparation.


\bibitem{Pt1} I.\ F.\ Putnam,
{\emph{The $C^{*}$-algebras associated with minimal homeomorphisms
of the Cantor set}},
Pacific J.\ Math.\ {\textbf{136}}(1989), 329--353.


\bibitem{Tk1} M.\ Takesaki, 
{\emph{Theory of Operator Algebras I}}, 
Springer-Verlag, Berlin, 2002.


\bibitem{TomsWinter} A.\ S.\ Toms and W.\ Winter,
{\emph{Minimal dynamics and K-theoretic rigidity: Elliott's 
conjecture}},
to appear (arXiv:math.OA/0903.4133v1).


\bibitem{Zh} S.\ Zhang,
{\emph{Matricial structure and homotopy type of simple 
$C^{*}$-algebras with real rank zero}},
J.\ Operator Theory {\textbf{26}}(1991), 283--312.


\end{thebibliography}
\end{document}